\documentclass[11pt]{amsart}
\usepackage{subcaption}
                                      
\usepackage{newlfont}
\usepackage{amssymb}
\usepackage{amsmath}
\usepackage{amsthm}
\usepackage{latexsym}
\usepackage{amsthm}
\usepackage{stmaryrd}
\usepackage{floatflt}
\usepackage{verbatim}
\usepackage{float}
\usepackage{indentfirst}
\usepackage{pdflscape}

\usepackage{tikz}
\usetikzlibrary{tikzmark}

\usepackage{epstopdf}
\usepackage{xcolor}
\usepackage[all,cmtip]{xy}
\usepackage{rotating}

\newcommand\scalemath[2]{\scalebox{#1}{\mbox{\ensuremath{\displaystyle #2}}}}
\newtheorem{thm}{Theorem}[section]  
\newtheorem{prop}[thm]{Proposition} 
\newtheorem{cor}[thm]{Corollary}
\newtheorem{lem}[thm]{Lemma}
\theoremstyle{definition}               
\newtheorem{defn}[thm]{Definition}
\newtheorem{rem}[thm]{Remark}
\newtheorem{ese}[thm]{Example} 
\theoremstyle{remark}    
\begin{document}
\title{Presentations of braid groups of type $A$ arising from $(m+2)$-angulations of regular polygons}
\author{Davide Morigi}
\maketitle
\pagestyle{myheadings}
\markboth{Davide Morigi}{Presentations of braid groups of type $A$ from $(m+2)$-angulations of regular polygons}

\begin{abstract}
We describe presentations of braid groups of type $A$ arising from coloured quivers of mutation type $A$. We show that these can be interpreted geometrically as generalised triangulations of regular polygons.
\end{abstract}

\texttt{Keywords:} mutation, braid groups, cluster categories.

\texttt{AMS 2010 classification number:} 13F60, 20F36

\section{Introduction}
The concept of $(m+2)$-angulation of a regular $(nm+2)$-sided polygon was studied in \cite{BM} by Baur and Marsh in 2006 to give a geometric description of the $m$-cluster category of type $A_{n-1}$. A few years later, Buan and Thomas introduced $m$-coloured quivers and $m$-coloured quiver mutations in \cite{BT1}, adapting the classical concept of quiver mutation given by Fomin and Zelevinsky in \cite{FZ} to the setting of higher cluster categories.  

Classically, the concepts of quiver of mutation type $A_{n-1}$ and triangulation of a regular $(n+2)$ polygon coincide (see \cite{FZ2}). In Section 2 of \cite{GM}, the authors associate a group to any quiver of mutation type $A_{n-1}$, that they show being isomorphic to the braid group of type $A_{n-1}$. They give an interpretation of such presentations, that already appeared in \cite{S}, in terms of triangulations of a regular $(n+2)$-gon. This construction allows them to interpret the aforementioned presentations from a cluster algebra point of view. 

This paper generalises the result of \cite{GM} just described, using the concepts introduced in \cite{BM} and \cite{BT1}. We propose a description of some presentations given in \cite{S} from an higher cluster category perspective. In particular, we describe presentations of the braid group of type $A_{n-1}$ arising from colured quiver mutation. The main theorem we will prove is the following.

\begin{thm}\label{thm_introduction}
Fix two integers $m,n\geq 1$. Let $Q$ be an $m$-coloured quiver of mutation type $\overrightarrow{A_{n-1}}$, with vertices $1,\ldots, n-1$. If $i\to j\in Q_1$, let $c_{ij}$ be its colour. Define the group $B_Q$ to be generated by $s_1,\ldots, s_{n-1}$ subject to the following relations.
\begin{enumerate}
\item $s_i s_j=s_js_i$ if there is no arrow between $i$ and $j$ (in either direction);
\item $s_i s_j s_i=s_js_is_j$ if there is a pair of arrows $\xymatrix@1{i\ar@<0.3ex>[r] & j\ar@<0.3ex>[l]}$;
\item $s_{i_1}s_{i_2}s_{i_3}s_{i_1}=s_{i_2}s_{i_3}s_{i_1}s_{i_2}=s_{i_3}s_{i_1}s_{i_2}s_{i_3}$ if 
$$
\xymatrixcolsep{1pc}\xymatrix{ & i_1\ar@<0.3ex>[dl]\ar@<0.3ex>[dr] & \\
i_2 \ar@<0.3ex>[rr]\ar@<0.3ex>[ur] & & i_3 \ar@<0.3ex>[ll]\ar@<0.3ex>[ul]}
$$
is a subquiver of $Q$, and $c_{i_1, i_2}+c_{i_2, i_3}+c_{i_3, i_1}=2m+1$.
\end{enumerate}
Then $B_Q$ is isomorphic to the braid group of type $\overrightarrow{A_{n-1}}$.
\end{thm}

We prove this result by giving explicit group isomorphisms between $B_Q$ and $B_{\mu_k(Q)}$, the group associated to the $m$-coloured quiver mutation of $Q$ at some vertex $k$. The proof relies on the geometric description of $m$-coloured quivers of mutation type $\overrightarrow{A_{n-1}}$ as $(m+2)$-angulations of a regular $(nm+2)$-sided polygon which was proved in \cite{T}. However, an explicit proof that the quiver associated to an $(m+2)$-angulation of a regular polygon is always of mutation type $\overrightarrow{A_{n-1}}$ is not given in \cite{T}. Therefore, we will go through the proof of this fact, and then tackle Theorem \ref{thm_introduction}.

The relations in Theorem \ref{thm_introduction} are the same given by Grant and Marsh in \cite{GM} for uncoloured quivers of mutation type $A$. However, in the uncoloured setting, the only possible full subquivers of a quiver $Q$ of mutation type $A$ are linear quivers and 3-cycles. In the coloured setting, if we fix an integer $m\geq 2$ and choose $n\in\mathbb{N}$ big enough, we can have as full subquivers of an $m$-coloured quiver $Q$ of mutation type $A_{n-1}$ as many $k$-cycles as we want, for $k=2,\ldots, m+2$. We will show that relations involving $k$-cycles also hold for all $k$, and follow from the ones in Theorem \ref{thm_introduction}.

The pictures in this paper have been done using the interactive mathematics software Geogebra \cite{GEO}.

\section{Coloured quivers}

Throughout the paper, $m,n\geq 1$ will be fixed integers. The following definitions can be found in \cite{BT1}. 

\begin{defn}\label{col_quiver}
An \textbf{$m$-coloured quiver} $Q$ consists of vertices $1,\ldots ,n$ and coloured arrows $i \xrightarrow{(c)} j$, where $c\in \{0,\ldots, m\}$. \\
Let $q_{ij}^{(c)}$ denote the number of arrows from $i$ to $j$ of colour $c$. We consider $m$-coloured quivers satisfying the following conditions.
\begin{itemize}
\item[(I)] \textit{No loops:} $q_{ii}^{(c)}=0$ for all $c=0,\ldots , m$.
\item[(II)] \textit{Monochromaticity:} $q_{ij}^{(c)}\neq 0$ implies $q_{ij}^{(c')}=0$ for all $c'\neq c$.
\item[(III)] \textit{Skew-symmetry:} $q_{ij}^{(c)}=q_{ji}^{(m-c)}$ for all $c=0,\ldots, m$.
\end{itemize}
\end{defn}

\begin{rem}\label{rem_std_quiver_interpretation}
A standard quiver (i.e., a quiver without colours on the arrows) shall be interpreted as a 1-coloured quiver whose arrows have colour 1. For example, a quiver of type $A_3$ with the usual orientation of the arrows will be interpreted as the following 1-coloured quiver.
$$
\xymatrix@1{1\ar[r] & 2\ar[r] & 3} \hspace{0.6cm} \leadsto \hspace{0.6cm} \xymatrix@1{1\ar@<0.3ex>[r]^{(1)} & 2\ar@<0.3ex>[l]^{(0)}\ar@<0.3ex>[r]^{(1)} & 3\ar@<0.3ex>[l]^{(0)}}
$$
\end{rem}

\begin{rem}
Notice that, because of the skew symmetry property, if an $m$-coloured quiver $Q$ has an arrow $i\stackrel{(c)}{\to} j$, $c\in\{0,\ldots, m\}$, then it also has an arrow $j\stackrel{(m-c)}{\to}i$. \\ 
Thus, we may sometimes draw
$$
i\stackrel{(c)}{\to} j
$$
instead of 
$$
\xymatrix@1{i\ar@<0.4ex>[r]^{(c)} & j\ar@<0.4ex>[l]^{(m-c)}}.
$$
\end{rem}

\begin{ese}
Let $m=3$, $n=5$. Then an example of 3-coloured  quiver on vertices $1,\ldots, 5$ is the following.
$$
\xymatrixcolsep{3pc}\xymatrix{1\ar@<0.4ex>[r]\ar@<-0.5pt>@{}[r]^{\scalemath{0.5}{(1)}} & 3 \ar@<0.4ex>[l]\ar@<-0.5pt>@{}[l]^{\scalemath{0.5}{(2)}} \ar@<0.4ex>[r]\ar@<-0.5pt>@{}[r]^{\scalemath{0.5}{(0)}} & 2\ar@<0.4ex>[l]\ar@<-0.5pt>@{}[l]^{\scalemath{0.5}{(3)}}\ar@<0.4ex>[rr]\ar@<-0.5pt>@{}[rr]^{\scalemath{0.5}{(1)}}\ar@<0.4ex>[dr]\ar@<-0.5pt>@{}[dr]^<<<<<<<<<{\scalemath{0.5}{(2)}} & & 4\ar@<0.4ex>[ll]\ar@<-0.5pt>@{}[ll]^{\scalemath{0.5}{(2)}}\ar@<0.4ex>[dl]\ar@<-0.5pt>@{}[dl]^{\scalemath{0.5}{(0)}} \\
& & & 5\ar@<0.4ex>[ul]\ar@<-0.5pt>@{}[ul]^{\scalemath{0.5}{(1)}}\ar@<0.4ex>[ur]\ar@<-0.5pt>@{}[ur]^<<<<<<<{\scalemath{0.5}{(3)}} }
$$
We might draw this as
$$
\xymatrix{1\ar@<0.4ex>[r]\ar@<-0.5pt>@{}[r]^{\scalemath{0.5}{(1)}} & 3  \ar@<0.4ex>[r]\ar@<-0.5pt>@{}[r]^{\scalemath{0.5}{(0)}} & 2\ar@<0.4ex>[rr]\ar@<-0.5pt>@{}[rr]^{\scalemath{0.5}{(1)}} & & 4\ar@<0.4ex>[dl]\ar@<-0.5pt>@{}[dl]^{\scalemath{0.5}{(0)}} \\
& & & 5\ar@<0.4ex>[ul]\ar@<-0.5pt>@{}[ul]^{\scalemath{0.5}{(1)}} }.
$$
\end{ese}

We now introduce the concept of mutation of an $m$-coloured quiver at a vertex $k$.  

\begin{defn}\label{defn_quiv_mut}
Let $Q$ be an $m$-coloured quiver, and $k\in \{1,\ldots, n\}$ one of its vertices. The $m$-\textbf{coloured quiver mutation} of $Q$ at vertex $k$ is the $m$-coloured quiver $\widetilde{Q}=\mu_k (Q)$ defined by
$$
\tilde{q}_{ij}^{(c)}=\begin{cases}
q_{ij}^{(c+1)} & \text{ if } k=i \\
q_{ij}^{(c-1)} & \text{ if } k=j \\
\max\scalemath{0.85}{\{0,q_{ij}^{(c)}-\sum_{t\neq c}q_{ij}^{(t)}+(q_{ik}^{(c)}-q_{ik}^{(c-1)})q_{kj}^{(0)}+q_{ik}^{(m)}(q_{kj}^{(c)}-q_{kj}^{(c+1)})\}} & \text{ if } i\neq k\neq j
\end{cases}
$$
where we set $q_{ij}^{(m+1)}=q_{ij}^{(0)}$ and $q_{ij}^{(-1)}=q_{ij}^{(m)}$.
\end{defn}

\begin{rem}
We will later associate to $m$-coloured quivers of some type a generalised triangulation of a regular polygon. In this interpretation, the quiver mutation defined above corresponds to the counterclockwise rotation of a diagonal in the polygon. 

If we replace all the occurrences of $c+1$ by $c-1$ (and viceversa) in the definition of $\tilde{q}_{ij}^{(c)}$ given above, then we would get a definition of $m$-coloured quiver mutation that is inverse to Definition \ref{defn_quiv_mut}. With this alternative definition, mutations should be interpreted as clockwise rotations of a diagonal in the associated generalised triangulation of a regular polygon.
\end{rem}

\begin{rem}
Let $Q$ be a standard (uncoloured) quiver. If we interpret $Q$ as a 1-coloured quiver as in Remark \ref{rem_std_quiver_interpretation}, then it is easy to check that the definition of 1-coloured quiver mutation given above agrees with the standard definition of quiver mutation given in literature (see \cite{FZ}, \S 4).
\end{rem}

The above definition is slightly different to the one given in \cite{BT1}. Indeed, in \cite{BT1} the authors swap the conditions $k=i$ and $k=j$. Our choice will be justified later (see Remark \ref{rem_correct_defn} and Section \ref{conn_gulations_quivers}).

One can give an equivalent definition of coloured quiver mutation, that is in general easier to work with. 

\begin{prop}\label{prop_equiv_col_quiv_mut}
Let $Q$ be an $m$-coloured quiver, and $k\in\{1,\ldots,n\}$ one of its vertices. Then the following algorithm correctly computes the $m$-coloured quiver mutation $\mu_k(Q)$ of $Q$ at vertex $k$. The colours -1 and $m+1$ that could arise shall be interpreted as $m$ and 0, respectively.
\begin{itemize}
\item[Step 1.] Add 1 to the colour of the arrows going into $k$, and subtract 1 to the colour of the arrows going out of $k$.
\item[Step 2.] For each of the following type of arrows
$$
\xymatrix@1{i\ar@<1ex>[r]^{(c)} & k\ar[l]^{(m-c)}\ar@<1ex>[r]^{(0)} & j \ar[l]^{(m)}},
$$
with $i\neq j$ and $c\neq m$, add the pair of arrows
$$
\xymatrix@1{i\ar@<1ex>[r]^{(c)} & j\ar[l]^{(m-c)}}.
$$
\item[Step 3.] If the graph obtained violates the monochromaticity property (II) of Definition \ref{col_quiver} because for some pair of vertices $i$ and $j$ there are arrows from $i$ to $j$ which have more than one different colour, cancel the same number of arrows of each colour, until property (II) is satisfied.
\end{itemize}
\end{prop}

\begin{rem}\label{rem_correct_defn}
The above Proposition can be found in Section 10 of \cite{BT1}. However, the authors don't impose the condition $c\neq m$ in Step 2. This is likely a mistake for the following two reasons:
\begin{itemize}
\item Let
$$
Q=\xymatrix@1{i\ar@<1ex>[r]^{(m)} & k\ar[l]^{(0)}\ar@<1ex>[r]^{(0)} & j \ar[l]^{(m)}}, \hspace{1cm} i\neq j.
$$
If we muatate at vertex $k$, Definition \ref{defn_quiv_mut} yields $\tilde{q}_{ij}^{(m)}=0$ and hence there should be no pair of arrows $\xymatrix{i\ar@<0.5ex>[r] & j\ar@<0.5ex>[l]} $ in $\mu_k(Q)$ while, according to Section 10 of \cite{BT1}, there should be a pair of arrows $\xymatrix@1{i\ar@<0.5ex>[r]^{(m)} & j \ar@<0.5ex>[l]^{(0)}}$.
\item The second reason will be fully explained in Section \ref{conn_gulations_quivers}. There, we will associate to some $m$-coloured quivers a combinatorial object, called $(m+2)$-angulation of an $(mn+2)$-gon. On such $(m+2)$-angulations we will define an operation, that we want to be compatible with the quiver mutation defined above. One can check that, in order for this to happen, the condition $c\neq m$ is crucial.
\end{itemize} 
\end{rem}

\begin{ese}
Let $m=2$ and let $Q$ be the following quiver.
$$
\xymatrix{ & 2\ar[dr]\ar@<-3pt>@{}[dr]^{\scalemath{0.5}{(0)}} & \\
1\ar[ur] \ar@<-3pt>@{}[ur]^{\scalemath{0.5}{(1)}} \ar[rr]\ar@<-3pt>@{}[rr]^{\scalemath{0.5}{(2)}} & & 3}
$$

Applying Steps 1. and 2. of Proposition \ref{prop_equiv_col_quiv_mut}  for $k=2$ yields the quiver
$$
\xymatrix{& 2\ar[dr]\ar@<-3pt>@{}[dr]^{\scalebox{0.5}{(2)}} & \\
1\ar[ur]\ar@<-3pt>@{}[ur]^{\scalebox{0.5}{(2)}} \ar@/^/[rr]^{\scalebox{0.5}{(2)}} \ar@/_/[rr]_{\scalebox{0.5}{(1)}} & & 3}
$$
However, the pair of vertices 1 and 3 violate the monochromaticity property of coloured quivers. Thus, if we apply Step 3. of Proposition \ref{prop_equiv_col_quiv_mut}, we get that the mutation of $Q$ at vertex 2 is 
$$
\mu_2 (Q)=\xymatrix@1{1\ar[r]^{(2)} & 2\ar[r]^{(2)} & 3}.
$$
\end{ese}

\begin{defn}\label{defn_quivAn1}
We call the following $m$-coloured quiver $\overrightarrow{A_{n-1}}$.
$$
\xymatrix@1{1\ar@<0.3ex>[r]\ar@<-1pt>@{}[r]^{\scalemath{0.5}{(0)}} & 2\ar@<0.3ex>[l]\ar@<-1pt>@{}[l]^{\scalemath{0.5}{(m)}}\ar@<0.3ex>[r]\ar@<-1pt>@{}[r]^<<<<<{\scalemath{0.5}{(0)}} & \cdots\ar@<0.3ex>[l]\ar@<-1pt>@{}[l]^>>>>>{\scalemath{0.5}{(m)}} \ar@<0.3ex>[r]\ar@<-1pt>@{}[r]^<<<<<{\scalemath{0.5}{(0)}}& (n-2)\ar@<0.3ex>[l]\ar@<-1pt>@{}[l]^>>>>>{\scalemath{0.5}{(m)}}\ar@<0.3ex>[r]\ar@<-1pt>@{}[r]^{\scalemath{0.5}{(0)}} & (n-1)\ar@<0.3ex>[l]\ar@<-1pt>@{}[l]^{\scalemath{0.5}{(m)}}}
$$
\end{defn}

\begin{defn}
An $m$-coloured quiver $Q$ is called \textbf{of mutation type $\overrightarrow{A_{n-1}}$} if $Q=\mu_{i_1}\ldots\mu_{i_{\ell}}(\overrightarrow{A_{n-1}})$ for some $i_1,\ldots, i_{\ell}\in\{1,\ldots, n-1\}$.
\end{defn}

\section{$(m+2)$-angulations of a regular polygon}
Recall $n,m\geq 1$ are fixed integers.

In the following, $\Pi$ will denote a regular $nm+2$ sided polygon, with vertices numbered clockwise from 1 to $nm+2$. 

We introduce the concept of $(m+2)$-angulation of $\Pi$. In the following we will see that this is closely related to $m$-coloured quivers of mutation type $\overrightarrow{A_{n-1}}$.

\begin{defn}\label{defn_m_gulation}
\begin{itemize}
\item A \textbf{diagonal} of $\Pi$ is a pair $(i,j)$ with $i\neq j$, $i,j\in\{1,\ldots, nm+2\}$. The diagonal $(i,j)$ will be interpreted the same as the diagonal $(j,i)$.
\item Two diagonals $(i_1,j_1)$, $(i_2, j_2)$ with $i_1<j_1$ and $i_2<j_2$ of $\Pi$ are called \textbf{intersecting} if $i_1<i_2<j_1<j_2$ or $i_2<i_1<j_2<j_1$.
\item An \textbf{$m$-diagonal} of $\Pi$ is a diagonal of $\Pi$ of the form $(i,i+jm+1)$ for some $i\in \{1,\ldots, nm+2\}$ and $j\in\{1,\ldots , n-1\}$.
\item An \textbf{$(m+2)$-angulation} of $\Pi$ is a maximal collection of non intersecting $m$-diagonals.
\end{itemize}
\end{defn}

\begin{rem}
One can give a geometrical interpretation the first three items of Definition \ref{defn_m_gulation}. 
\begin{itemize}
\item A \textbf{diagonal} of $\Pi$ is a segment connecting two distinct vertices of $\Pi$.
\item Two diagonals of $\Pi$ are called \textbf{intersecting} if they have a common point inside $\Pi$.
\item An \textbf{$m$-diagonal} of $\Pi$ is a diagonal of $\Pi$ which divides $\Pi$ into an $(mj+2)$-gon and an $(m(n-j)+2)$-gon for some $j=1,\ldots, n-1$. 
\end{itemize}

We will mainly use the geometric description of the concepts defined above, since it allows us to draw pictures. However, one might use a purely combinatorial approach as well.
\end{rem}

The following results follow directly from the definition of $(m+2)$-angulation.

\begin{prop}\label{prop_properties_m2_guls}
Let $\Delta$ be an $(m+2)$-angulation of $\Pi$. Then the following hold.
\begin{enumerate}
\item The number of $m$-diagonals of $\Delta$ is $n-1$.
\item $\Delta$ defines $n$ distinct $(m+2)$-gons $P_1,\ldots , P_n$ whose union is $\Pi$. 
\item If $n\geq 2$ then, for each $\gamma\in \Delta$ there are exactly two $(m+2)$-gons $P_{\gamma}^{(1)}$, $P_{\gamma}^{(2)}$ determined by $\Delta$ that have $\gamma$ as an edge. We call $P_{\gamma}$ the $(2m+2)$-gon given by the union of $P_{\gamma}^{(1)}$ and $P_{\gamma}^{(2)}$.
\end{enumerate}
\end{prop}
\begin{proof}
The first two statements can be proved by induction on $n$.
\begin{itemize}
\item If $n=1$, then $\Pi$ is an $(m+2)$-gon. Hence $\Delta$ is empty, that is, the number of $m$-diagonals of $\Delta$ is 0, and $\Pi$ is trivially union of one $(m+2)$-gon, that is $\Pi$ itself.
\item If $n>1$, let $\gamma\in \Delta$ be an $m$-diagonal. Then $\gamma$ divides $\Pi$ into an $(mj+2)$-gon $P$ and a $m(n-j)+2$-gon $P'$ for some $j=1,\ldots, n-1$.
\begin{itemize}
\item By induction hypothesis, we have that $\Delta$ induces an $(m+2)$-angulation on $P$ (resp. $P'$) consisting of $j-1$ $m$-diagonals (resp. $n-j-1$ $m$-diagonals), say $\delta_1,\ldots, \delta_{j-1}$ (resp. $\delta_{j+1}, \ldots, \delta_{n-1}$). Hence $\Delta=\{\delta_1,\ldots, \delta_{j-1}, \gamma, \delta_{j+1}, \ldots, \delta_{n-1}\}$, so the first statement follows.
\item By induction hypothesis, $P$ can be written as disjoint union of $j$ $(m+2)$-gons, say $P_1,\ldots, P_j$, while $P'$ can be written as union of $n-j$ $(m+2)$-gons, say $P_{j+1}, \ldots, P_n$. Hence $\Delta$ can be written as union of the $(m+2)$-gons $P_1,\ldots, P_n$, and so we get the second statement.
\end{itemize}  
\end{itemize}
As for (3), we know that $\gamma$ divides $\Pi$ into an $(mj+2)$-gon $\Pi_1$ and an $(m(n-j)+2)$-gon $\Pi_2$ for some $j\in \{1,\ldots, n-1\}$. Let $\Delta_i$ be the restriction of $\Delta$ to $\Pi_i$, for $i=1,2$. By (2), $\Delta_1$ (resp. $\Delta_2$) defines $j$ (resp. $n-j$) $(m+2)$-gons whose union is $\Pi_1$ (resp. $\Pi_2$). Therefore $P_{\gamma}^{(1)}$ and $P_{\gamma}^{(2)}$ are the $(m+2)$-gons that have $\gamma$ as an edge in $\Pi_1$ and $\Pi_2$, respectively.
\end{proof}

\begin{defn}\label{def_rotation}
Let $\Delta$ be an $(m+2)$-angulation of $\Pi$, and $\gamma\in\Delta$ be an $m$-diagonal. Let $P_{\gamma}$ be the $(2m+2)$-gon introduced in Proposition \ref{prop_properties_m2_guls}. Let $\{a_0,\ldots, a_{2m+1}\}$ be the vertices of $P_{\gamma}$, with $a_0<a_1<\ldots <a_{2m+1}$. Then the \textbf{mutation} of $\Delta$ at $\gamma$ is the following $(m+2)$-angulation of $\Pi$
$$
r_{\gamma}(\Delta)=(\Delta\setminus \gamma)\cup \{\sigma (\gamma)\},
$$
where, if $\gamma=(a_i,a_{i+m+1})$ for some $i=0,\ldots,m$, then 
$$
\sigma(\gamma)=(a_{i-1 \text{ (mod 2m+2)}},a_{i+m \text{ (mod 2m+2)}}).
$$ 
\end{defn}

\begin{rem}
The $(m+2)$-angulation $r_{\gamma}(\Delta)$ can be geometrically interpreted as the $(m+2)$-angulation of $\Pi$ obtained by fixing all the diagonals of $\Delta$ except from $\gamma$, that is rotated counterclockwise inside the $(2m+2)$-gon $P_{\gamma}$.
\end{rem}

\begin{prop}\label{rem_order_rotation}
Let $\Delta$ be an $(m+2)$-angulation of $\Pi$, $\gamma\in\Delta$. Then $r_{\gamma}$ has order $m+1$.
\end{prop}
\begin{proof}
Let $\{a_0,\ldots, a_{2m+1}\}$ be the vertices of $P_{\gamma}$, with $a_0<a_1<\ldots < a_{2m+1}$, and let $i\in\{0,\ldots, m\}$ be such that $\gamma=(a_i,a_{i+m+1})$. Definition \ref{def_rotation} implies that
$$
\sigma^j(\gamma)=(a_{i-j \text{ (mod 2m+2)}}, a_{i+m+1-j \text{ (mod 2m+2)}})
$$
for all $j\geq 0$. Hence 
$$
\sigma^{m+1}(\gamma)=\sigma^j(\gamma)=(a_{i-(m+1) \text{ (mod 2m+2)}}, a_{i \text{ (mod 2m+2)}})=(a_{i+m+1},a_i)
$$ 
is the $m$-diagonal $\gamma$ by Definition \ref{defn_m_gulation}. It is also straightforward to check that $\sigma^j(\gamma)\neq \gamma$ for all $1\leq j<m+1$. Thus $\sigma$ has order $m+1$. The definition of $r_{\gamma}$ given in Definition \ref{def_rotation} implies that
$$
r_{\gamma}^j(\Delta)=(\Delta\setminus\gamma)\cup \{\sigma^j(\gamma)\}
$$
for all $j\geq 1$. Therefore, also $r_{\gamma}$ has order $m+1$.
\end{proof}

\begin{ese}\label{ese_4_gulations}
Let $m=2$, $n=5$. Consider the diagonal $\gamma$ in the 4-angulation $\Delta$ of the regular dodecagon given below. The hexagon $P_{\gamma}$ is coloured in picture, and the result of applying $r_{\gamma}$ is displayed. 
\begin{center}
\includegraphics[scale=0.29]{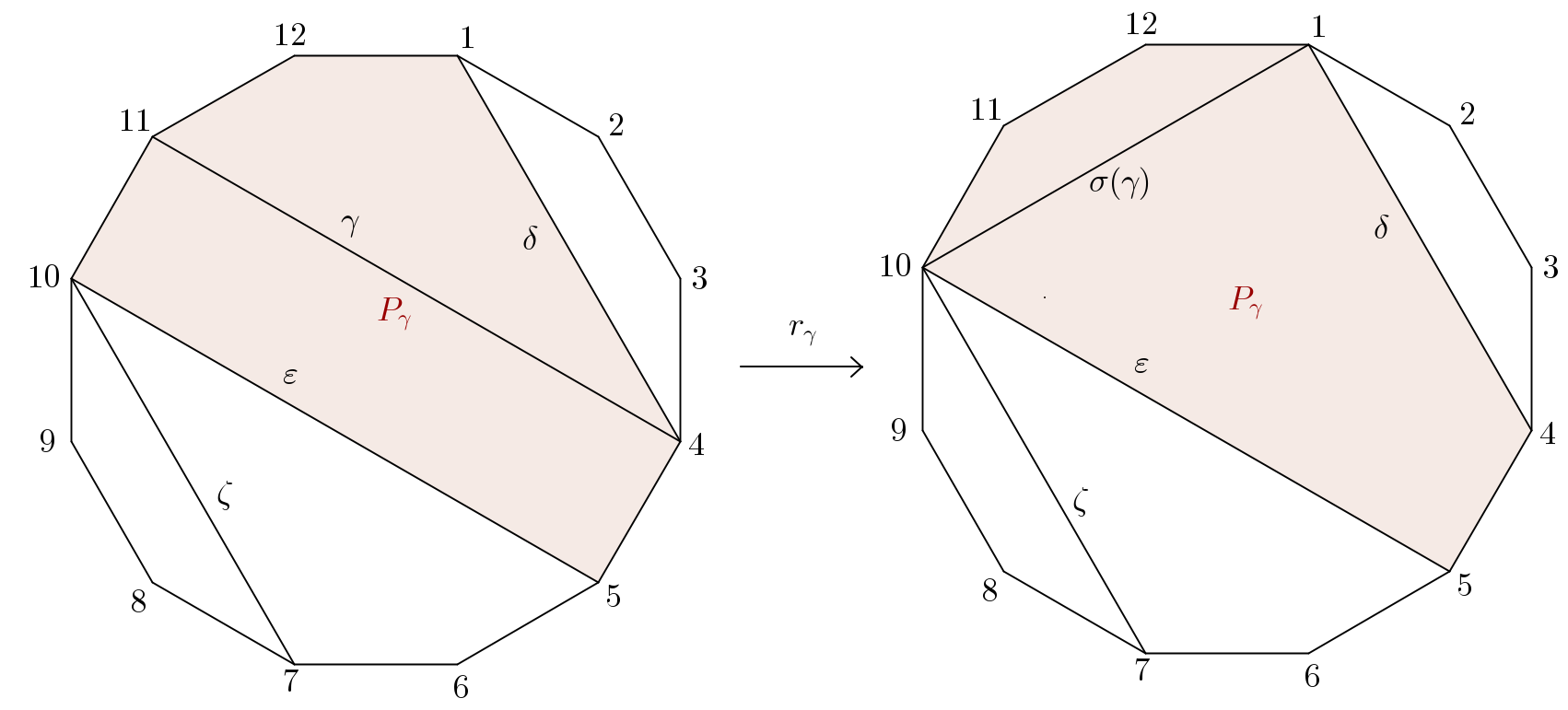}
\end{center}
The application of $r_{\gamma}$ might hence be thought as a counterclockwise rotation of $\gamma$ inside the hexagon $P_{\gamma}$.
\end{ese}

\section{Connection between $(m+2)$-angulations and $m$-coloured quivers}\label{conn_gulations_quivers}

In this section we introduce a map $\Psi$, that associates an $m$-coloured quiver to an $(m+2)$-angulation of $\Pi$. This construction will be done in a way such that the two concepts of mutation we introduced (at a vertex for an $m$-coloured quiver, and at a diagonal for an $(m+2)$-angulation of $\Pi$) commute. This result, that we will prove in detail, is used in \cite{T} to show that $\Psi$ induces a bijection between $(m+2)$-angulations of $\Pi$ and $m$-coloured quivers of mutation type $\overrightarrow{A_{n-1}}$.

\begin{defn}\label{defn_ass_quiv_to_m_gul}
Let $\Delta$ be an $(m+2)$-angulation of $\Pi$. We associate to $\Delta$ an $m$-coloured quiver $\Psi(\Delta)$ as follows.
\begin{itemize}
\item The vertices of $\Psi(\Delta )$ are the $m$-diagonals of $\Delta$.
\item If $\gamma, \delta$ are $m$-diagonals of $\Delta$ which are edges of some $(m+2)$-gon in the $(m+2)$-angulation $\Delta$, then $\Psi(\Delta)$ has an arrow from $\gamma$ to $\delta$. In this case, the colour of the arrow is the number of edges forming the segment of the boundary of the $(m+2)$-gon which lie between $\gamma$ and $\delta$, counterclockwise from $\gamma$ and clockwise from $\delta$.
\end{itemize}
\end{defn}

\begin{ese}
The 2-coloured quiver $\Psi(r_{\gamma}(\Delta))$ associated to the 4-angulation $r_{\gamma}(\Delta)$ in Example \ref{ese_4_gulations} is given by
$$
\xymatrix{ \sigma(\gamma)\ar@<0.3ex>[rr]\ar@<-1pt>@{}[rr]^{\scalemath{0.5}{(2)}}\ar@<0.3ex>[dr]\ar@<-1pt>@{}[dr]^{\scalemath{0.5}{(0)}} & & \delta\ar@<0.3ex>[dl]\ar@<-1pt>@{}[dl]^{\scalemath{0.5}{(1)}}\ar@<0.3ex>[ll]\ar@<-1pt>@{}[ll]^{\scalemath{0.5}{(0)}}\\
 & \varepsilon\ar@<0.3ex>[ur]\ar@<-1pt>@{}[ur]^{\scalemath{0.5}{(1)}}\ar@<0.3ex>[ul]\ar@<-1pt>@{}[ul]^{\scalemath{0.5}{(2)}}\ar@<0.3ex>[dl]\ar@<-1pt>@{}[dl]^{\scalemath{0.5}{(0)}} & \\
\zeta\ar@<0.3ex>[ur]\ar@<-1pt>@{}[ur]^{\scalemath{0.5}{(2)}}  & & }
$$
\end{ese}

In the following, we will identify $\Pi$ with a circle with $mn+2$ marked points, numbered clockwise from 1 to $mn+2$.

\begin{defn}
We call $\widehat{A}_{n-1}$ the following $(m+2)$-angulation of $\Pi$.
\begin{center}
\includegraphics[scale=0.352]{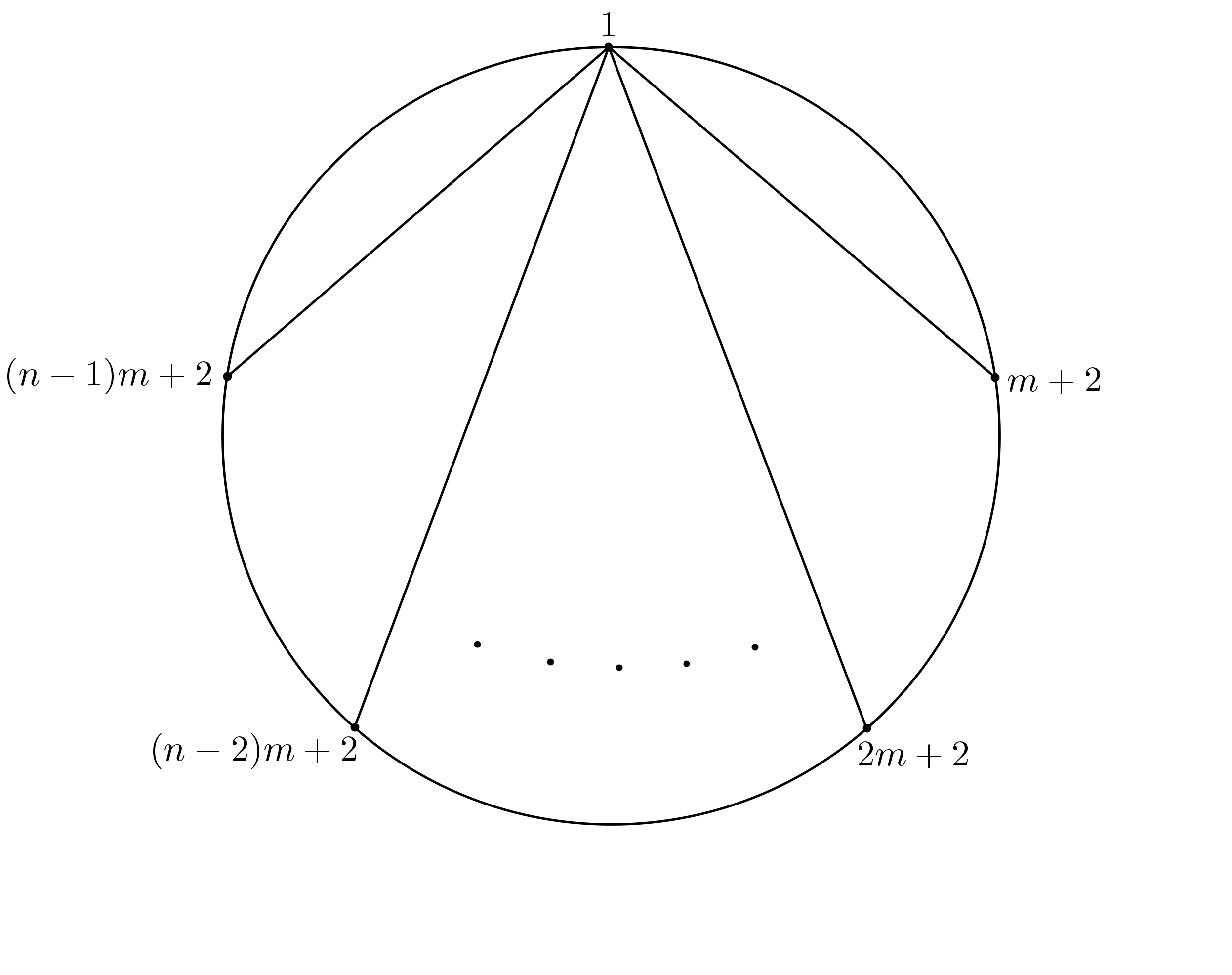}
\end{center}
\end{defn}

\begin{rem}\label{rem_connection_betweenAn}
Consider the $m$-coloured quiver $\overrightarrow{A_{n-1}}$ introduced in Definition \ref{defn_quivAn1}. Then
$$
\overrightarrow{A_{n-1}}=\Psi(\widehat{A}_{n-1})
$$
\end{rem}

We want to prove that every $(m+2)$-angulation of $\Pi$ can be obtained from $\widehat{A}_{n-1}$ by iteratively mutating it at some $m$-diagonals. In order to do this, we introduce the concept of distance between a diagonal of an $(m+2)$-angulation $\Delta$ of $\Pi$ and the vertex 1.

\begin{defn}
Let $\Delta$ be an $(m+2)$-angulation of $\Pi$, $\gamma\in \Delta$ one of its $m$-diagonals. We define the \textbf{distance} $d_{\Delta}(\gamma)$ between $\gamma$ and the vertex 1 of $\Pi$ as follows:
\begin{itemize}
\item $d_{\Delta}(\gamma)=0$ if $\gamma=(1,i)$ for some $i$;
\item $d_{\Delta}(\gamma)=k$ if $\gamma$ is an edge of a $(2m+2)$-gon $P_{\delta}$ defined in Proposition \ref{prop_properties_m2_guls}, for some $\delta\in \Delta$ with $d_{\Delta}(\delta)=k-1$. 
\end{itemize}
\end{defn}

\begin{ese}
Consider the $4$-angulation $r_{\gamma}(\Delta)$ of $\Pi$ given in Example \ref{ese_4_gulations}. Then $d_{r_{\gamma}(\Delta)}(\delta)=0=d_{r_{\gamma}(\Delta)}(\sigma (\gamma))$, $d_{r_{\gamma}(\Delta)}(\varepsilon)=1$, $d_{r_{\gamma}(\Delta)}(\zeta)=2$.
\end{ese}

\begin{lem}\label{lem_rotation_An_1}
Let $\Delta$ be an $(m+2)$-angulation of $\Pi$. Then there exist $m$-diagonals $\gamma_1,\ldots, \gamma_{\ell}$ such that
$$
\Delta=r_{\gamma_1}\ldots r_{\gamma_{\ell}}(\widehat{A}_{n-1}).
$$ 
\end{lem}
\begin{proof}
We prove the result by induction on $k=\max \{d_{\Delta} (\gamma) |\gamma\in\Delta\}$.
\begin{itemize}
\item If $k=0$, then $\Delta=\widehat{A}_{n-1}$.
\item Let $k>0$. Order the $m$-diagonals of $\Delta$ as
$$
\Delta=\{\gamma_{0,1}, \ldots, \gamma_{0,t_0}, \gamma_{1,1}, \ldots, \gamma_{1,t_1}, \ldots, \gamma_{k,1}, \ldots, \gamma_{k,t_k}\},
$$
where $d_{\Delta} (\gamma_{i,j})=i$ for all $i,j$. Fix $j\in\{1,\ldots, t_1\}$. Since $d_{\Delta}(\gamma_{1,j})=1$, then we know that the $(2m+2)$-gon $P_{\gamma_{1,j}}$ contains the vertex 1. Hence we can mutate the $m$-diagonal $\gamma_{1,j}$ a certain number of times, say $s_j$, so that $\sigma^{s_j}(\gamma_{1,j})$ has 1 as endpoint. Now, consider the following $(m+2)$-angulation of $\Pi$.
$$
\Delta '={\scriptstyle r_{\gamma_{1,t_1}}^{s_{t_1}}\ldots r_{\gamma_{1,1}}^{s_1}}(\Delta){\scriptstyle =\{\underbrace{\gamma_{0,1}}_{=\gamma_{0,1}'}, \ldots, \underbrace{\gamma_{0,t_0}}_{=\gamma_{0,t_0}'}, \underbrace{\sigma^{s_1}(\gamma_{1,1})}_{=\gamma_{1,1,}'}, \ldots, \underbrace{\sigma^{s_{t_1}}(\gamma_{1,t_1})}_{=\gamma_{1,t_1}'}, \ldots, \underbrace{\gamma_{k,1}}_{=\gamma_{k,1}'}, \ldots, \underbrace{\gamma_{k,t_k}}_{=\gamma_{k,t_k}'}\}}.
$$
By assumption we have $d_{\Delta'}(\gamma_{0,j}')=0$ for all $j$. 

Also, $d_{\Delta'}(\gamma_{i,j}')=i-1$ for $i=1,\ldots, k$ and all $j$. This can be shown by induction on $i\geq 1$. 
\begin{itemize}
\item If $i=1$, then by construction we know that all the $m$-diagonals $\gamma_{1,j}'=\sigma^{s_i}(\gamma_{1,j})$ have 1 as endpoint. This means that $d_{\Delta'}(\gamma_{1,j}')=0$ for all $j$.
\item Let $i>1$, and fix $j\in\{1,\ldots, t_i\}$. Then, by contruction, we can find an $m$-diagonal $\gamma_{i-1,\ell}'$ that is an edge of $P_{\gamma_{i,j}'}$. But $d_{\Delta'}(\gamma_{i-1,\ell}')=i-2$ by induction hypothesis, and thus $d_{\Delta'}(\gamma_{i,j}')=i-1$ by definition of distance $d_{\Delta'}$.
\end{itemize}
Hence $\max\{d_{\Delta '}(\gamma ') | \gamma ' \in \Delta '\}=k-1$. Thus, by induction we get
\begin{equation}\label{eq_widehat_An-1_Delta}
\widehat{A}_{n-1}=r_{\tilde{\gamma}_1}\ldots r_{\tilde{\gamma}_s}(\Delta')=r_{\tilde{\gamma}_1}\ldots r_{\tilde{\gamma}_s}r_{\gamma_{1,t_1}}^{s_{t_1}}\ldots r_{\gamma_{1,1}}^{s_1}(\Delta)
\end{equation}
for some $m$-diagonals $\tilde{\gamma}_1,\ldots, \tilde{\gamma}_s$.

But all mutations $r_{\gamma}$ have finite order $m+1$ by Proposition \ref{rem_order_rotation}, and hence we can invert all mutations in \eqref{eq_widehat_An-1_Delta}. Therefore we get the statement.
\end{itemize}
\end{proof}

\begin{defn}
Let $Q$ be an $m$-coloured quiver, $k\in Q_0$ one of its vertices. Define
\begin{itemize}
\item $\mathcal{N}_{Q,k}=\{i\in Q_0 | \text{ there are arrows } \xymatrix@1{k\ar@<0.4ex>[r] & i\ar@<0.4ex>[l]}\}$ to be the \textbf{neighbourhood of $k$ in $Q$}.
\item $F_{Q,k}$ as the full subquiver of $Q$ with vertex set $\mathcal{N}_{Q,k}\cup\{k\}$.
\end{itemize}
\end{defn}

\begin{defn}
\begin{itemize}
\item We call an $m$-coloured quiver $Q$  \textbf{complete} if, for every pair of vertices $i,j\in Q_0$, there is exactly one pair of arrows $\xymatrix@1{i\ar@<0.3ex>[r] & j\ar@<0.3ex>[l]}$.
\item Let $Q$, $Q'$ quivers with vertex sets $\{k,u_1,\ldots, u_i\}$ and $\{k,v_1,\ldots, v_j\}$ respectively, for some $i,j\geq 0$. Then the \textbf{glueing of $Q$ and $Q'$ at vertex $k$} is the quiver with vertex set $\{k,u_1,\ldots, u_i,v_1,\ldots, v_j\}$ and arrows $Q_1\cup Q_1'$.
\end{itemize}
\end{defn}

The following Lemma will help us understand how the $m$-coloured quiver associated to an $(m+2)$-angulation of $\Pi$ mutates at a vertex. Its proof follows by a straightforward application of Definition \ref{defn_ass_quiv_to_m_gul}.

\begin{lem}\label{lem_local_mutations}
Let $Q$ be an $m$-coloured quiver, $k\in Q_0$. Suppose that $F_{Q,k}$ has the following shape
$$
\xymatrixcolsep{0.4pc}\xymatrix{ & & u_a\ar@<0.3ex>[dll]\ar@<0.3ex>[ddl]\ar@<0.3ex>[ddr]\ar@<0.3ex>[drr] & & & v_1\ar@<0.3ex>[dl]\ar@<0.3ex>[rr]\ar@<0.3ex>[ddr]\ar@<0.3ex>[drrr] & & v_2\ar@<0.3ex>[ll]\ar@<0.3ex>[dlll]\ar@<0.3ex>[ddl]\ar@{.}[dr] & \\
u_{a-1}\ar@<0.3ex>[urr]\ar@<0.3ex>[drrr]\ar@<0.3ex>[rrrr]\ar@{.}[dr] & & & & k\ar@<0.3ex>[ull]\ar@<0.3ex>[llll]\ar@<0.3ex>[dlll]\ar@<0.3ex>[dl]\ar@<0.3ex>[ur]\ar@<0.3ex>[urrr]\ar@<0.3ex>[drr]\ar@<0.3ex>[rrrr] & & & & v_{b-1}\ar@<0.3ex>[llll]\ar@<0.3ex>[ulll]\ar@<0.3ex>[dll] \\
& u_2\ar@<0.3ex>[uur]\ar@<0.3ex>[rr]\ar@<0.3ex>[urrr] & & u_1\ar@<0.3ex>[uul]\ar@<0.3ex>[ulll]\ar@<0.3ex>[ll]\ar@<0.3ex>[ur] & & & v_b\ar@<0.3ex>[uul]\ar@<0.3ex>[ull]\ar@<0.3ex>[uur]\ar@<0.3ex>[urr] & &
}
$$
where the arrows can have any colour.

In words, suppose that $F_{Q,k}$ is the quiver obtained by glueing the complete quiver on $\{k,u_1,\ldots, u_a\}$ and the complete quiver on $\{k,v_1,\ldots, v_b\}$ at vertex $k$, for some $a,b\geq 0$.

For each pair of vertices $h,\ell\in \mathcal{N}_{Q,k}$ connected by an arrow, let $c_{h\ell}$ be the colour of the arrow $h\to \ell$. Suppose that $c_{k,u_a}<c_{k,u_{a-1}}<\ldots < c_{k,u_1}$ and $c_{k,v_b}< c_{k,v_{b-1}}<\ldots<c_{k,v_1}$. 

Then the mutation $\mu_k (Q)$ of $Q$ at vertex $k$ fixes the arrows (together with their colours) of $Q_1\setminus (F_{Q,k})_1$, and acts as follows on $F_{Q,k}$ (the blue vertices will denote those vertices that have moved, and red arrows will denote the new arrows):

\begin{itemize}
\item[a)] If $c_{k,u_a}\neq 0\neq c_{k,v_b}$, then $\mu_k(F_{Q,k})$ has the same shape as $F_{Q,k}$:
\[
\xymatrixcolsep{0.4pc}\xymatrix{ & & u_a\ar@<0.3ex>[dll]\ar@<0.3ex>[ddl]\ar@<0.3ex>[ddr]\ar@<0.3ex>[drr] & & & v_1\ar@<0.3ex>[dl]\ar@<0.3ex>[rr]\ar@<0.3ex>[ddr]\ar@<0.3ex>[drrr] & & v_2\ar@<0.3ex>[ll]\ar@<0.3ex>[dlll]\ar@<0.3ex>[ddl]\ar@{.}[dr] & \\
u_{a-1}\ar@<0.3ex>[urr]\ar@<0.3ex>[drrr]\ar@<0.3ex>[rrrr]\ar@{.}[dr] & & & & k\ar@<0.3ex>[ull]\ar@<0.3ex>[llll]\ar@<0.3ex>[dlll]\ar@<0.3ex>[dl]\ar@<0.3ex>[ur]\ar@<0.3ex>[urrr]\ar@<0.3ex>[drr]\ar@<0.3ex>[rrrr] & & & & v_{b-1}\ar@<0.3ex>[llll]\ar@<0.3ex>[ulll]\ar@<0.3ex>[dll] \\
& u_2\ar@<0.3ex>[uur]\ar@<0.3ex>[rr]\ar@<0.3ex>[urrr] & & u_1\ar@<0.3ex>[uul]\ar@<0.3ex>[ulll]\ar@<0.3ex>[ll]\ar@<0.3ex>[ur] & & & v_b\ar@<0.3ex>[uul]\ar@<0.3ex>[ull]\ar@<0.3ex>[uur]\ar@<0.3ex>[urr] & &
}
\] 
and the colours of the arrows change as follows.
\begin{align*}
\widetilde{c}_{u_i,u_j} = & c_{u_i,u_j}, & i,j\in\{1,\ldots , a\} \\
\widetilde{c}_{v_i,v_j} = & c_{v_i,v_j}, & i,j\in\{1,\ldots , b\} \\
\widetilde{c}_{k,u_i} = & c_{k,u_i} -1, & i\in\{1,\ldots, a\} \\
\widetilde{c}_{k,v_i}= & c_{k,v_i}-1, & i\in\{1,\ldots, b\}
\end{align*}

\item[b)] 
If $c_{k,u_a}=0$, $c_{k,v_b}\neq 0$, then $\mu_k(F_{Q,k})$ is given by
\[
\xymatrixcolsep{0.4pc}\xymatrix{ & & u_{a-1}\ar@<0.3ex>[dll]\ar@<0.3ex>[ddl]\ar@<0.3ex>[ddr]\ar@<0.3ex>[drr] & & & {\color{blue}u_a}\ar@[red]@<0.3ex>[dl]\ar@[red]@<0.3ex>[rr]\ar@[red]@<0.3ex>[ddr]\ar@[red]@<0.3ex>[drrr] & & v_1\ar@[red]@<0.3ex>[ll]\ar@<0.3ex>[dlll]\ar@<0.3ex>[ddl]\ar@{.}[dr] & \\
u_{a-2}\ar@<0.3ex>[urr]\ar@<0.3ex>[drrr]\ar@<0.3ex>[rrrr]\ar@{.}[dr] & & & & k\ar@<0.3ex>[ull]\ar@<0.3ex>[llll]\ar@<0.3ex>[dlll]\ar@<0.3ex>[dl]\ar@[red]@<0.3ex>[ur]\ar@<0.3ex>[urrr]\ar@<0.3ex>[drr]\ar@<0.3ex>[rrrr] & & & & v_{b-1}\ar@<0.3ex>[llll]\ar@[red]@<0.3ex>[ulll]\ar@<0.3ex>[dll] \\
& u_2\ar@<0.3ex>[uur]\ar@<0.3ex>[rr]\ar@<0.3ex>[urrr] & & u_1\ar@<0.3ex>[uul]\ar@<0.3ex>[ulll]\ar@<0.3ex>[ll]\ar@<0.3ex>[ur] & & & v_b\ar@[red]@<0.3ex>[uul]\ar@<0.3ex>[ull]\ar@<0.3ex>[uur]\ar@<0.3ex>[urr] & &
}
\]
and the colours of the arrows change as follows.
\begin{align*}
\widetilde{c}_{u_i,u_j}= & c_{u_i,u_j}, & i,j\in\{1,\ldots, a-1\} \\
\widetilde{c}_{v_i,v_j} = & c_{v_i,v_j}, & i,j\in\{1,\ldots, b\} \\
\widetilde{c}_{k,u_i} = & c_{k,u_i} -1, & i\in\{1,\ldots, a\} \\
\widetilde{c}_{k,v_i} = & c_{k,v_i} -1, & i\in\{1,\ldots b\} \\ 
\widetilde{c}_{v_i,u_a} = & c_{v_i,k}, & i\in\{1,\ldots, b\}
\end{align*}

\item[c)] If $c_{k,u_a}\neq 0, c_{k,v_b}=0$, then $\mu_k(F_{Q,k})$ is given by 
\[
\xymatrixcolsep{0.4pc}\xymatrix{ & & u_a\ar@<0.3ex>[dll]\ar@<0.3ex>[ddl]\ar@[red]@<0.3ex>[ddr]\ar@<0.3ex>[drr] & & & v_1\ar@<0.3ex>[dl]\ar@<0.3ex>[rr]\ar@<0.3ex>[ddr]\ar@<0.3ex>[drrr] & & v_2\ar@<0.3ex>[ll]\ar@<0.3ex>[dlll]\ar@<0.3ex>[ddl]\ar@{.}[dr] & \\
u_{a-1}\ar@<0.3ex>[urr]\ar@[red]@<0.3ex>[drrr]\ar@<0.3ex>[rrrr]\ar@{.}[dr] & & & & k\ar@<0.3ex>[ull]\ar@<0.3ex>[llll]\ar@<0.3ex>[dlll]\ar@[red]@<0.3ex>[dl]\ar@<0.3ex>[ur]\ar@<0.3ex>[urrr]\ar@<0.3ex>[drr]\ar@<0.3ex>[rrrr] & & & & v_{b-2}\ar@<0.3ex>[llll]\ar@<0.3ex>[ulll]\ar@<0.3ex>[dll] \\
& u_1\ar@<0.3ex>[uur]\ar@[red]@<0.3ex>[rr]\ar@<0.3ex>[urrr] & & {\color{blue}v_b}\ar@[red]@<0.3ex>[uul]\ar@[red]@<0.3ex>[ulll]\ar@[red]@<0.3ex>[ll]\ar@[red]@<0.3ex>[ur] & & & v_{b-1}\ar@<0.3ex>[uul]\ar@<0.3ex>[ull]\ar@<0.3ex>[uur]\ar@<0.3ex>[urr] & &
}
\]
and the colours of the arrows change as follows.
\begin{align*}
\widetilde{c}_{u_i,u_j} = & c_{u_i,u_j}, & i,j\in\{1,\ldots, a\} \\
\widetilde{c}_{v_i,v_j} = & c_{v_i,v_j}, & i,j\in\{1,\ldots, b-1\} \\
\widetilde{c}_{k,u_i} = & c_{k,u_i}-1, & i\in\{1,\ldots, a\} \\
\widetilde{c}_{k,v_i} = & c_{k,v_i} -1, & i\in\{1,\ldots, b\} \\
\widetilde{c}_{u_i,v_b} = & c_{u_i,k}, & i\in\{1,\ldots, a\}
\end{align*}

\item[d)] If $c_{k,u_a}=0=c_{k,v_b}$, then $\mu_k(F_{Q,k})$ is given by
\[\xymatrixcolsep{0.4pc}
\xymatrixcolsep{0.4pc}\xymatrix{ & & u_{a-1}\ar@<0.3ex>[dll]\ar@<0.3ex>[ddl]\ar@[red]@<0.3ex>[ddr]\ar@<0.3ex>[drr] & & & {\color{blue}u_a}\ar@[red]@<0.3ex>[dl]\ar@[red]@<0.3ex>[rr]\ar@[red]@<0.3ex>[ddr]\ar@[red]@<0.3ex>[drrr] & & v_1\ar@[red]@<0.3ex>[ll]\ar@<0.3ex>[dlll]\ar@<0.3ex>[ddl]\ar@{.}[dr] & \\
u_{a-2}\ar@<0.3ex>[urr]\ar@[red]@<0.3ex>[drrr]\ar@<0.3ex>[rrrr]\ar@{.}[dr] & & & & k\ar@<0.3ex>[ull]\ar@<0.3ex>[llll]\ar@<0.3ex>[dlll]\ar@[red]@<0.3ex>[dl]\ar@[red]@<0.3ex>[ur]\ar@<0.3ex>[urrr]\ar@<0.3ex>[drr]\ar@<0.3ex>[rrrr] & & & & v_{b-2}\ar@<0.3ex>[llll]\ar@[red]@<0.3ex>[ulll]\ar@<0.3ex>[dll] \\
& u_1\ar@<0.3ex>[uur]\ar@[red]@<0.3ex>[rr]\ar@<0.3ex>[urrr] & & {\color{blue}v_b}\ar@[red]@<0.3ex>[uul]\ar@[red]@<0.3ex>[ulll]\ar@[red]@<0.3ex>[ll]\ar@[red]@<0.3ex>[ur] & & & v_{b-1}\ar@[red]@<0.3ex>[uul]\ar@<0.3ex>[ull]\ar@<0.3ex>[uur]\ar@<0.3ex>[urr] & &
}
\]
and the colours of the arrows change as follows.
\begin{align*}
\widetilde{c}_{u_i,u_j}= & c_{u_i,u_j}, & i,j\in\{1,\ldots, a-1\} \\
\widetilde{c}_{v_i,v_j} = & c_{v_i,v_j}, & i,j\in\{1,\ldots, b-1\} \\
\widetilde{c}_{k,u_i} = & c_{k,u_i}-1, & i\in\{1,\ldots, a\} \\
\widetilde{c}_{k,v_i} = & c_{k,v_i} -1, & i\in\{1,\ldots, b\} \\
\widetilde{c}_{u_i, v_b} = & c_{u_i, k}, & i\in\{1,\ldots, a-1\} \\
\widetilde{c}_{v_i, u_a} = & c_{v_i, k}, & i\in\{1,\ldots, b-1\}
\end{align*}

\end{itemize}
\end{lem}

The following proposition is stated in \cite{BT1}, after Proposition 11.1. We give an explicit proof of this fact that relies on the action of $\mu_k$ on $m$-coloured quivers described in Lemma \ref{lem_local_mutations}.

\begin{prop}\label{prop_comm_mutations}
Let $\Delta$ be an $(m+2)$-angulation of $\Pi$, and consider an $m$-diagonal $\gamma\in\Delta$. Then
$$
\Psi (r_{\gamma}(\Delta))=\mu_{\gamma}(\Psi(\Delta)).
$$
\end{prop}
\begin{proof}
Let $P_{\gamma}$ be the regular $(2m+2)$-gon defined in Proposition \ref{prop_properties_m2_guls}. Let $P_{\gamma}^{(1)}$, $P_{\gamma}^{(2)}$ be the two $(m+2)$-gons induced by $\Delta$ having $\gamma$ as an edge, so that $P_{\gamma}=P_{\gamma}^{(1)}\cup P_{\gamma}^{(2)}$. 

Let $\delta_1^{(1)},\ldots, \delta_a^{(1)}, \delta_{a+1}^{(1)}=\gamma$ (resp. $\delta_1^{(2)},\ldots, \delta_b^{(2)}, \delta_{b+1}^{(2)}=\gamma$) be the edges of $P_{\gamma}^{(1)}$ (resp. of $P_{\gamma}^{(2)}$) that are $m$-diagonals of $\Delta$, ordered clockwise, for some $a,b\geq 0$. 

Let $\ell_i^{(1)}$ (resp. $\ell_i^{(2)}$) be the number of edges forming the segment of the boundary of $P_{\gamma}^{(1)}$ (resp. $P_{\gamma}^{(2)}$) which lies between $\delta_i^{(1)}$ and $\delta_{i-1}^{(1)}$ (resp. between $\delta_i^{(2)}$ and $\delta_{i-1}^{(2)}$), counterclockwise from $\delta_i^{(1)}$ (resp. $\delta_i^{(2)}$) for $i=1,\ldots, a+1$ (resp. $i=1,\ldots,b+1$). By convention, we set $\delta_{0}^{(1)}=\delta_{0}^{(2)}=\gamma$. See Figure \ref{fig_local_gulation} for a geometric description.

\begin{figure*}[h!]
\centering
\includegraphics[scale=0.5]{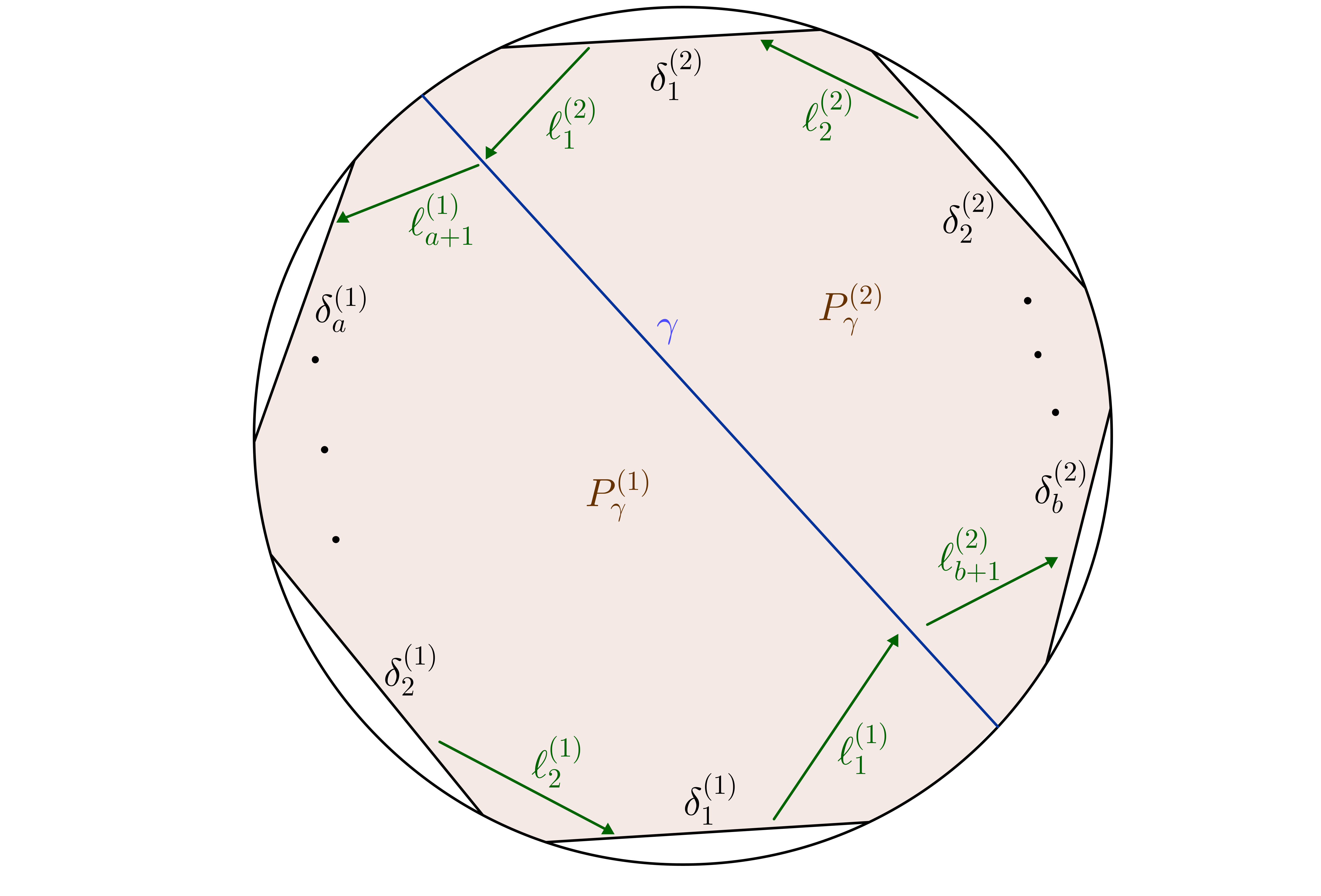}
\caption{Local $(m+2)$-angulation $\Delta$ around $\gamma$}
\label{fig_local_gulation}
\end{figure*}
Using Definition \ref{defn_ass_quiv_to_m_gul} we get that $F_{\Psi(\Delta),\gamma}$ is the following quiver,
$$
\xymatrixcolsep{0.5pc}\xymatrix{ & & \delta_a^{(1)}\ar@<0.3ex>[dll]\ar@<0.3ex>[ddl]\ar@<0.3ex>[ddr]\ar@<0.3ex>[drr] & & & \delta_1^{(2)}\ar@<0.3ex>[dl]\ar@<0.3ex>[rr]\ar@<0.3ex>[ddr]\ar@<0.3ex>[drrr] & & \delta_2^{(2)}\ar@<0.3ex>[ll]\ar@<0.3ex>[dlll]\ar@<0.3ex>[ddl]\ar@{.}[dr] & \\
\delta_{a-1}^{(1)}\ar@<0.3ex>[urr]\ar@<0.3ex>[drrr]\ar@<0.3ex>[rrrr]\ar@{.}[dr] & & & & \gamma\ar@<0.3ex>[ull]\ar@<0.3ex>[llll]\ar@<0.3ex>[dlll]\ar@<0.3ex>[dl]\ar@<0.3ex>[ur]\ar@<0.3ex>[urrr]\ar@<0.3ex>[drr]\ar@<0.3ex>[rrrr] & & & & \delta_{b-1}^{(2)}\ar@<0.3ex>[llll]\ar@<0.3ex>[ulll]\ar@<0.3ex>[dll] \\
& \delta_2^{(1)}\ar@<0.3ex>[uur]\ar@<0.3ex>[rr]\ar@<0.3ex>[urrr] & & \delta_1^{(1)}\ar@<0.3ex>[uul]\ar@<0.3ex>[ulll]\ar@<0.3ex>[ll]\ar@<0.3ex>[ur] & & & \delta_b^{(2)}\ar@<0.3ex>[uul]\ar@<0.3ex>[ull]\ar@<0.3ex>[uur]\ar@<0.3ex>[urr] & &
}
$$
where the colour of the arrows $\delta_i^{(1)}\to \delta_j^{(1)}$ is $c_{i,j}^{(1)}=\ell_i^{(1)} +\ell_{i-1}^{(1)}+\ldots +\ell_{j+1}^{(1)}+(i-j-1)$ for all $1\leq j<i\leq a+1$, and the colour of the arrows $\delta_i^{(2)}\to \delta_j^{(2)}$ is $c_{i,j}^{(2)}=\ell_i^{(2)} +\ell_{i-1}^{(2)}+\ldots + \ell_{j+1}^{(2)}+(i-j+1)$ for all $1\leq j<i\leq b+1$. 

Notice that 
$$ c_{a+1,a}^{(1)}<  c_{a+1,a-1}^{(1)}<\ldots < c_{a+1,1}^{(1)},\hspace{0.8cm} c_{b+1,b}^{(2)}<c_{b+1,b-1}^{(2)}<\ldots < c_{b+1,1}^{(2)}.
$$

Therefore we can apply Lemma \ref{lem_local_mutations}, that gives us all the possible mutations of $\Psi(\Delta)$ at vertex $\gamma$. In particular, $\mu_{\gamma}$ fixes the arrows (together with their colours) of $\Psi(\Delta)\setminus F_{\Psi(\Delta),\gamma}$, and it acts on $F_{\Psi(\Delta), \gamma}$ as explained in Lemma \ref{lem_local_mutations}. This computes $\mu_{\gamma}(\Psi(\Delta)).$

We now compute $r_{\gamma}(\Delta)$. First of all notice that $r_{\gamma}$ fixes all the $m$-diagonals of $\Delta$ apart from $\gamma$. Hence we may as well just represent how $r_{\gamma}$ acts on $\Delta$ locally around $\gamma$, that is, on $P_{\gamma}^{(1)}$ and $P_{\gamma}^{(2)}$. We have four possible cases, depending on the values of $\ell_{a+1}^{(1)}$ and $\ell_{b+1}^{(2)}$ (see Figure \ref{mut_poly_at_gamma} in the next page).

Now, applying the definition of $\Psi$ (see Definition \ref{defn_ass_quiv_to_m_gul}), one gets that the quiver $\Psi(r_{\gamma}(\Delta))$ associated to $r_{\gamma}(\Delta)$ is exactly the one given in Lemma \ref{lem_local_mutations}.

For example, if $\ell_{a+1}^{(1)}=0, \ell_{b+1}^{(2)}\neq 0$, one gets the commutative diagram in Figure \ref{commutative_diagram}. 

Hence we get the statement.
\end{proof}

\begin{figure*}[h!]
\centering
\begin{subfigure}[h]{0.54\textwidth}
\includegraphics[width=\textwidth]{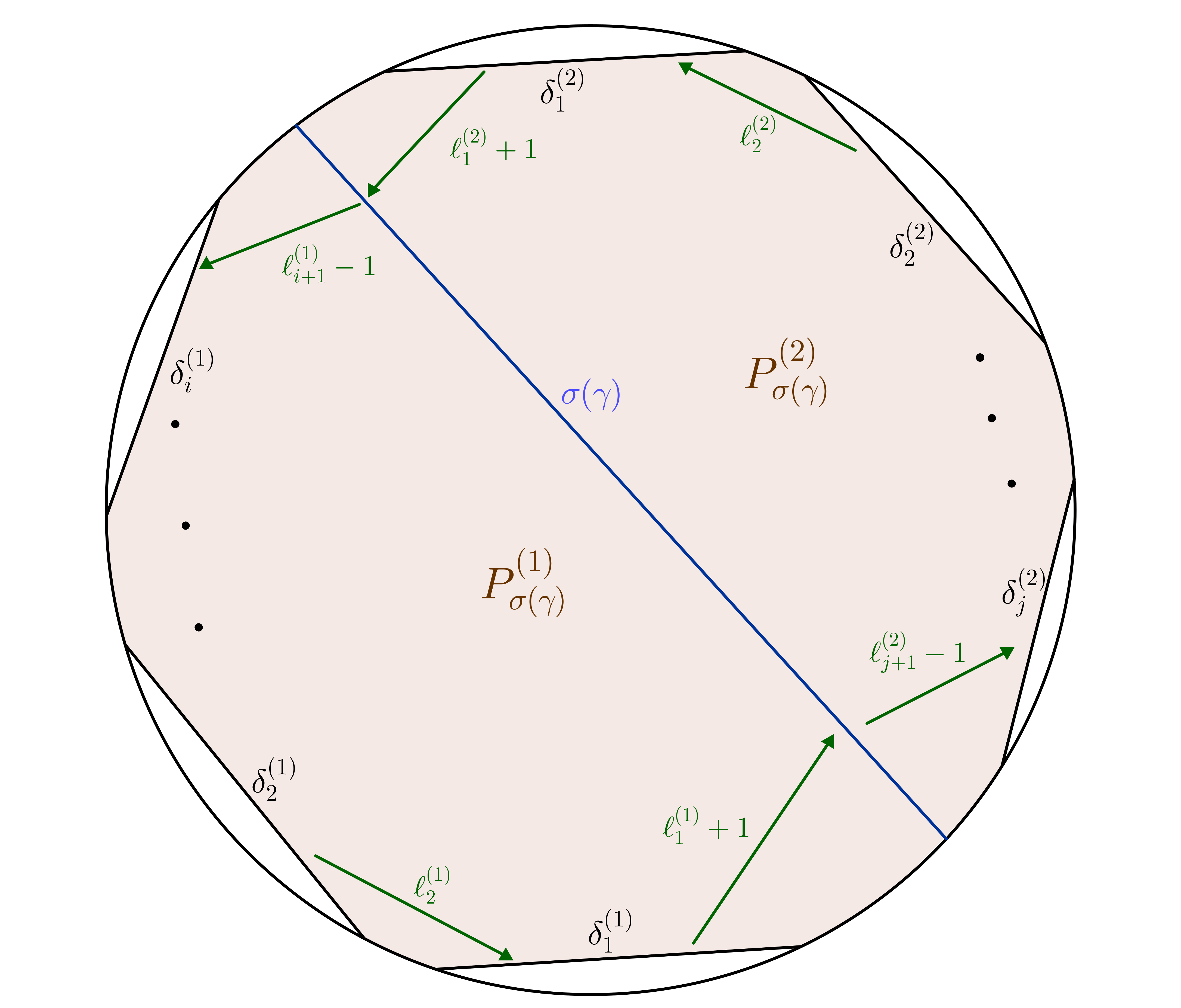}
\caption{$\ell_{a+1}^{(1)}\neq 0\neq \ell_{b+1}^{(2)}$}
\end{subfigure}\begin{subfigure}[h]{0.7\textwidth}\includegraphics[width=\textwidth]{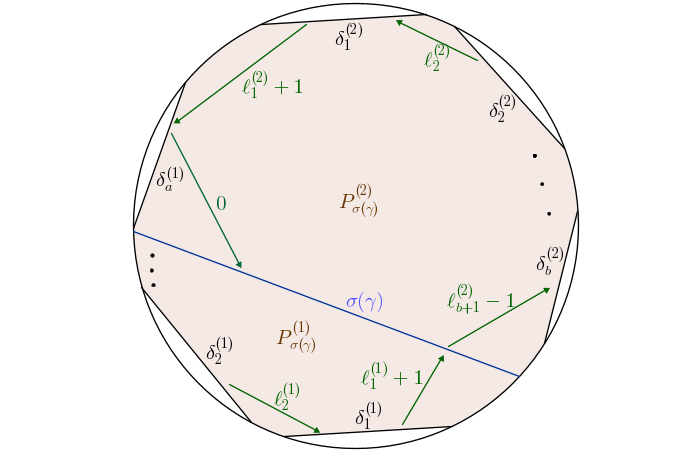}
\caption{$\ell_{a+1}^{(1)}= 0, \ell_{b+1}^{(2)}\neq 0$}
\end{subfigure}
\begin{subfigure}[h]{0.58\textwidth}
\includegraphics[width=\textwidth]{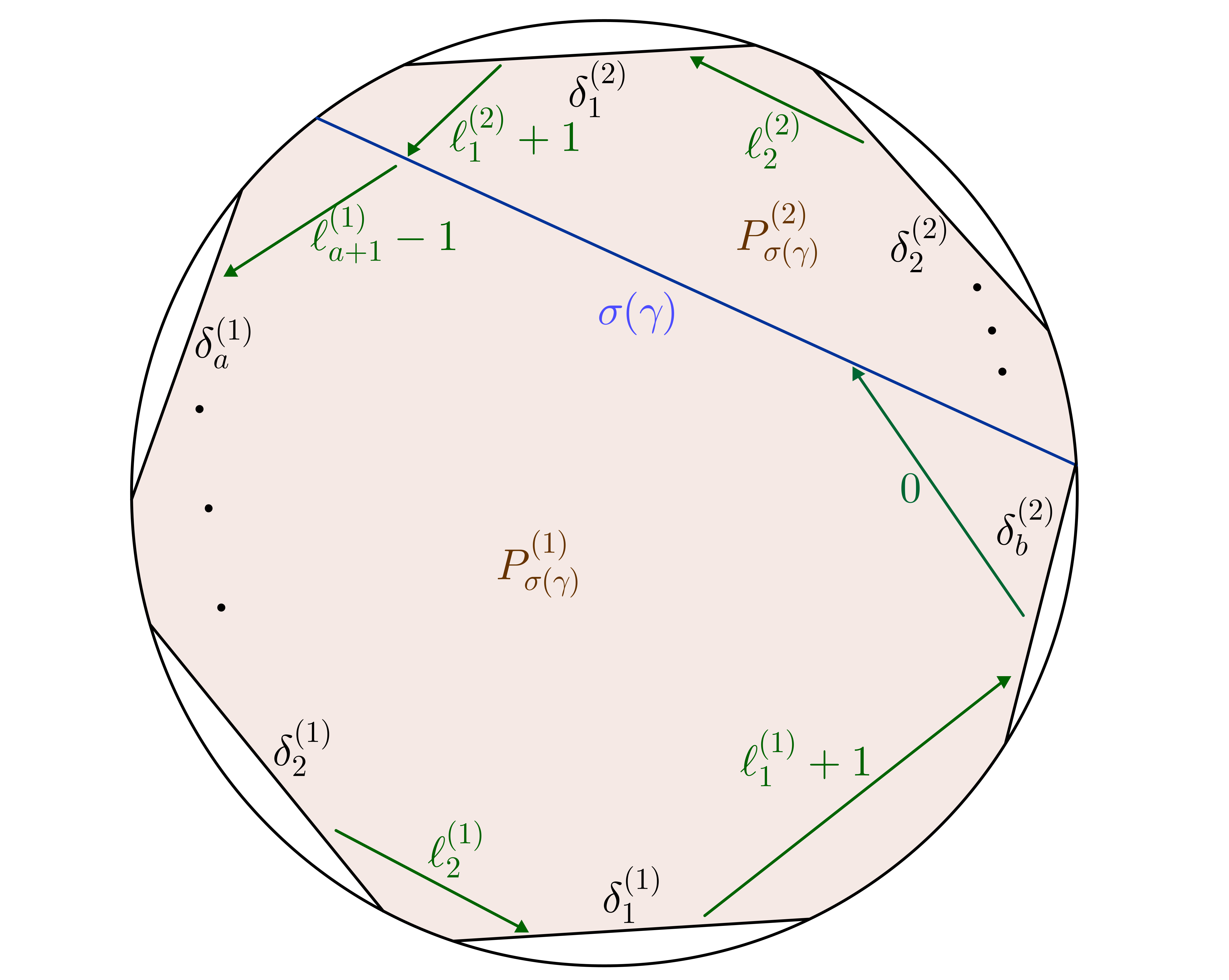}
\caption{$\ell_{a+1}^{(1)}\neq 0, \ell_{b+1}^{(2)}=0$}
\end{subfigure}\begin{subfigure}[h]{0.55\textwidth}
\includegraphics[width=\textwidth]{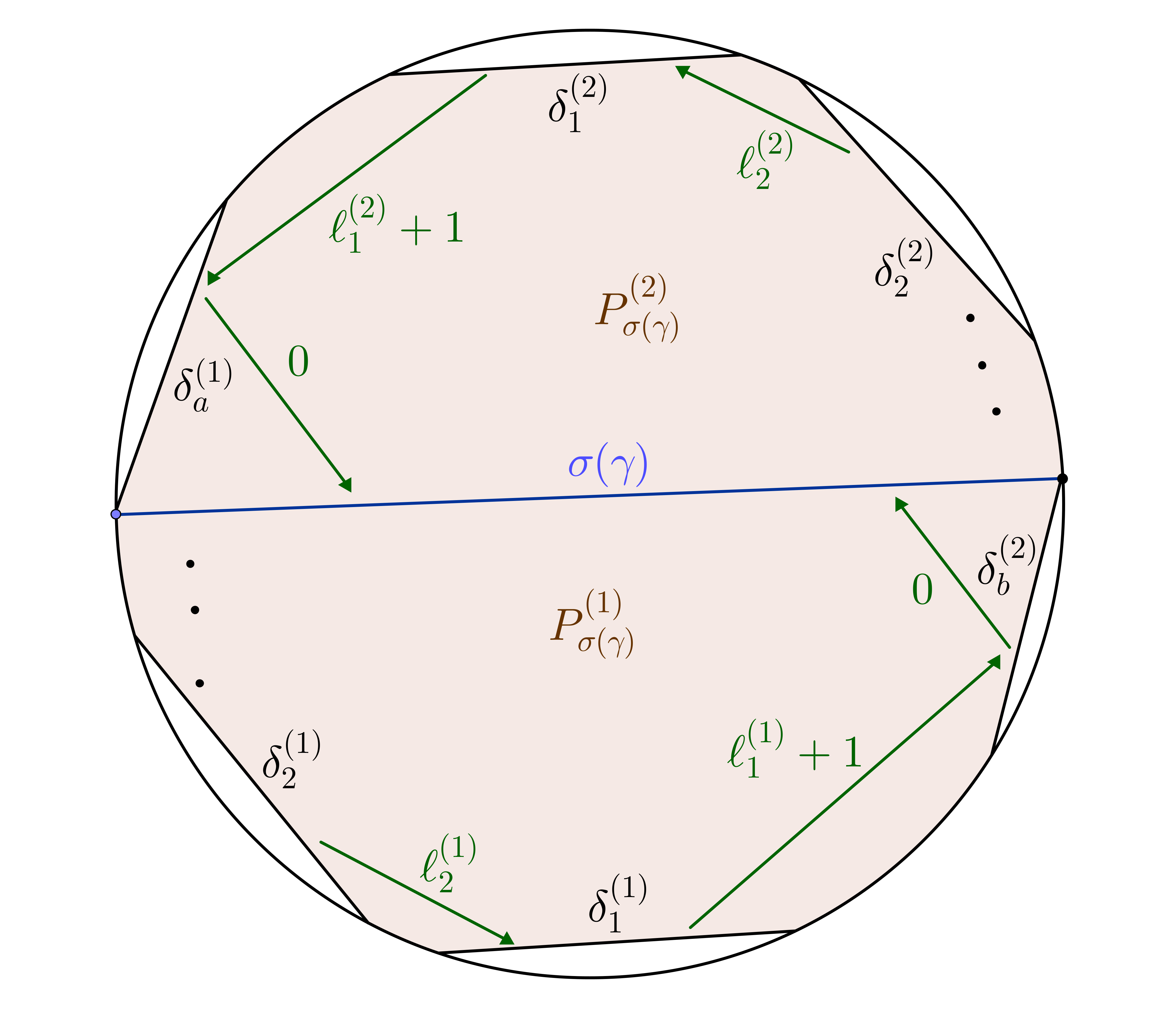}
\caption{$\ell_{a+1}^{(1)}= 0=\ell_{b+1}^{(2)}$}
\end{subfigure}
\caption{Mutations $r_{\gamma}(\Delta)$ at $\gamma$}
\label{mut_poly_at_gamma}
\end{figure*}

\begin{landscape}
\begin{figure*}[h!]
\includegraphics[scale=0.46]{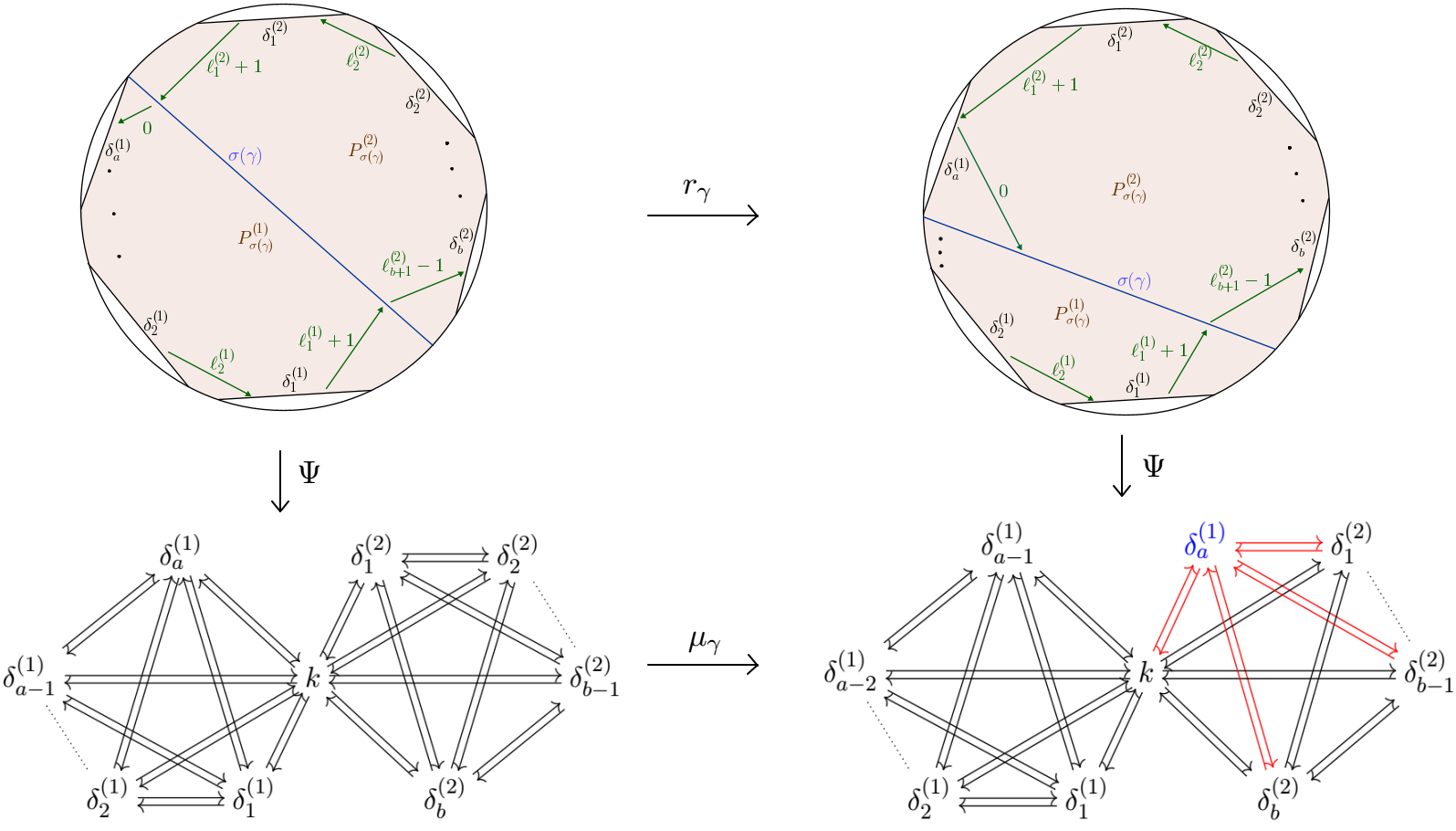}
\caption{Commutative diagram for $\ell_{i+1}^{(1)}=0$, $\ell_{j+1}^{(2)}\neq 0$}
\label{commutative_diagram}
\end{figure*}
\end{landscape}

\begin{rem}\label{rem_wrong_defn_quiv_mut}
The above Proposition is the second reason why we defined $m$-coloured quiver mutation in a slightly different way compared to \cite{BT1}. Consider the following example:
$$
Q=\xymatrix@1{1\ar@<0.3ex>[r]\ar@<-0.3pt>@{}[r]^{\scalemath{0.5}{(2)}} & 2\ar@<0.3ex>[l]\ar@<-0.3pt>@{}[l]^{\scalemath{0.5}{(0)}}\ar@<0.3ex>[r]\ar@<-0.3pt>@{}[r]^{\scalemath{0.5}{(0)}} & 3\ar@<0.3ex>[l]\ar@<-0.3pt>@{}[l]^{\scalemath{0.5}{(2)}}} \hspace{2cm} m=2.
$$
Then, if we mutate at vertex 2 using Proposition 10.1 in \cite{BT1}, we get that $\mu_2(Q)$ is the following quiver.
$$
\xymatrixcolsep{3pc}\xymatrix{& 2\ar@<0.5ex>[dl]\ar@<-0.3pt>@{}[dl]^{\scalemath{0.5}{(2)}}\ar@<0.5ex>[dr]\ar@<-0.3pt>@{}[dr]^{\scalemath{0.5}{(2)}} & \\
1\ar@<0.5ex>[rr]\ar@<-0.3pt>@{}[rr]^{\scalemath{0.5}{(2)}}\ar@<0.5ex>[ur]\ar@<-0.3pt>@{}[ur]^{\scalemath{0.5}{(0)}} & & 3\ar@<0.5ex>[ll]\ar@<-0.3pt>@{}[ll]^{\scalemath{0.5}{(0)}} \ar@<0.5ex>[ul]\ar@<-0.3pt>@{}[ul]^{\scalemath{0.5}{(0)}}}
$$
However, the diagram in Figure \ref{non_commutative_diagram} does not commute. 

\begin{figure*}[h]
\centering\includegraphics[scale=0.4]{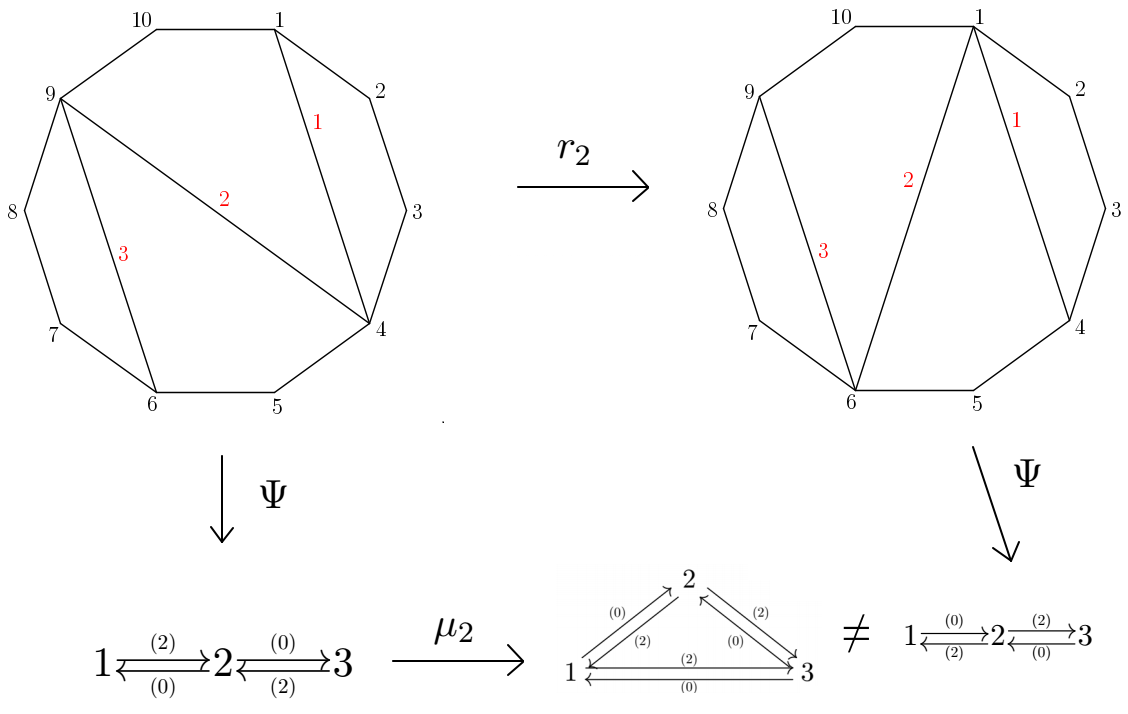}
\caption{Diagram obtained from Proposition 10.1 in \cite{BT1}}
\label{non_commutative_diagram}
\end{figure*} 
This suggests that the correct definition of $m$-coloured quiver mutation is the one in Proposition \ref{prop_equiv_col_quiv_mut}, that in this example correctly gives
$$
\mu_2(Q)=\xymatrix@1{1\ar@<0.3ex>[r]\ar@<-0.3pt>@{}[r]^{\scalemath{0.5}{(0)}} & 2\ar@<0.5ex>[l]\ar@<-0.3pt>@{}[l]^{\scalemath{0.5}{(2)}}\ar@<0.5ex>[r]\ar@<-0.3pt>@{}[r]^{\scalemath{0.5}{(2)}} & 3\ar@<0.5ex>[l]\ar@<-0.3pt>@{}[l]^{\scalemath{0.5}{(0)}}}.
$$
\end{rem}

Proposition \ref{prop_comm_mutations} is used in the following Theorem, whose proof can be found in Section 4 of \cite{T}.

\begin{thm}\label{thm_bij_col_quivs_m_guls}
Let $\mathcal{M}_{n-1,m}$ be the set of $m$-coloured quivers of mutation type $\overrightarrow{A_{n-1}}$, and $\mathcal{T}_{m,n}$ the set of $(m+2)$-angulations of $\Pi$. Then the map $\Psi$ introduced in Definition \ref{defn_ass_quiv_to_m_gul} induces a bijection
$$
\Psi:(\mathcal{T}_{m,n}/\sim )\to \mathcal{M}_{n-1,m}/\mathfrak{S}_{n-1}, 
$$
where $\Delta\sim \Delta '$ if $\Delta '$ can be obtained rotating $\Delta$, and the symmetric group $\mathfrak{S}_{n-1}$ on $n-1$ letters acts on an $m$-coloured quiver by permuting its vertices. 
\end{thm}

\begin{cor}\label{cor_order_col_quiv_mut_type_A}
Let $Q$ be an $m$-coloured quiver of mutation type $\overrightarrow{A_{n-1}}$ and $k\in Q_0$. Then $\mu_k$, the mutation of $Q$ at vertex $k$, has order $m+1$.
\end{cor}
\begin{proof}
This follows directly from Proposition \ref{rem_order_rotation} and Theorem \ref{thm_bij_col_quivs_m_guls}.
\end{proof}

\begin{rem}
Corollary \ref{cor_order_col_quiv_mut_type_A} holds for arbitrary quivers in the case $m=1$. Equivalently, for all (uncoloured) quivers $Q$ and $k\in Q_0$, then $\mu_k$ has order $m+1=2$.

This is not in general true for $m>1$. 

For example, let $m=2$ and consider the following 2-coloured quiver $Q$, and its repeated mutations at vertex $k=2$ obtained applying repeatidly Proposition \ref{prop_equiv_col_quiv_mut}.
$$
\vspace{0.3cm}
\begin{array}{cccccccc}
Q & \to & \mu_2(Q) & \to & \mu_2^2(Q) & \to & \mu_2^3(Q) \\ 
\xymatrixcolsep{1.5pc}\xymatrix{ & 2\ar@<0.4ex>[dr]\ar@<-0.5pt>@{}[dr]^{\scalemath{0.5}{(0)}}\ar@<0.4ex>[dl]\ar@{}[dl]^{\scalemath{0.5}{(2)}} & \\
1\ar@<0.4ex>[rr]\ar@<-0.5pt>@{}[rr]^{\scalemath{0.5}{(0)}}\ar@<0.4ex>[ur]\ar@<-0.5pt>@{}[ur]^{\scalemath{0.5}{(0)}} & & 3\ar@<0.4ex>[ll]\ar@<-0.5pt>@{}[ll]^{\scalemath{0.5}{(2)}}\ar@<0.4ex>[ul]\ar@{}[ul]^{\scalemath{0.5}{(2)}}}
 & \raisebox{-17pt}{$\mapsto$}  & \raisebox{-15pt}{\xymatrixcolsep{1.5pc}\xymatrix@1{1 \ar@<0.4ex>[r]\ar@<-0.5pt>@{}[r]^{\scalemath{0.5}{(1)}} & 2 \ar@<0.4ex>[l]\ar@<-0.5pt>@{}[l]^{\scalemath{0.5}{(1)}} \ar@<0.4ex>[r]\ar@<-0.5pt>@{}[r]^{\scalemath{0.5}{(2)}} & 3\ar@<0.4ex>[l]\ar@<-0.5pt>@{}[l]^{\scalemath{0.5}{(0)}}}} & \raisebox{-17pt}{$\mapsto$} & \raisebox{-15pt}{\xymatrixcolsep{1.5pc}\xymatrix@1{1 \ar@<0.4ex>[r]\ar@<-0.5pt>@{}[r]^{\scalemath{0.5}{(2)}} & 2 \ar@<0.4ex>[l]\ar@<-0.5pt>@{}[l]^{\scalemath{0.5}{(0)}} \ar@<0.4ex>[r]\ar@<-0.5pt>@{}[r]^{\scalemath{0.5}{(1)}} & 3\ar@<0.4ex>[l]\ar@<-0.5pt>@{}[l]^{\scalemath{0.5}{(1)}}}} & \raisebox{-17pt}{$\mapsto$} & \xymatrixcolsep{1.5pc}\xymatrix{ & 2\ar@<0.4ex>[dr]\ar@<-0.5pt>@{}[dr]^{\scalemath{0.5}{(0)}}\ar@<0.4ex>[dl]\ar@{}[dl]^{\scalemath{0.5}{(2)}} & \\
1\ar@<0.4ex>[rr]\ar@<-0.5pt>@{}[rr]^{\scalemath{0.5}{(1)}}\ar@<0.4ex>[ur]\ar@<-0.5pt>@{}[ur]^{\scalemath{0.5}{(0)}} & & 3\ar@<0.4ex>[ll]\ar@<-0.5pt>@{}[ll]^{\scalemath{0.5}{(1)}}\ar@<0.4ex>[ul]\ar@{}[ul]^{\scalemath{0.5}{(2)}}} 
\end{array}
$$
Hence we have $\mu_2^3(Q)\neq Q$.

Therefore, for an arbitrary $m$-coloured quiver $Q$ and $k\in Q_0$ we have $\mu_k^{m+1}(Q)\neq Q$, and the condition that $Q$ is of mutation type $\overrightarrow{A_{n-1}}$ is crucial.
\end{rem}

\begin{cor}
Let $Q$ be an $m$-coloured quiver of mutation type $\overrightarrow{A_{n-1}}$, and let $i,j\in Q_0$. Then there is at most one arrow $\xymatrix@1{i\ar[r] & j}$. 
\end{cor}
\begin{proof}
Thanks to Theorem \ref{thm_bij_col_quivs_m_guls} this is equivalent to saying that, for each pair of $m$-diagonals $\gamma , \gamma '$ in an $(m+2)$-angulation of $\Pi$, there is at most one $(m+2)$-gon inside $\Pi$ having both $\gamma$ and $\gamma '$ as edges, that is trivially true.
\end{proof}

Let $\Delta$ be an $(m+2)$-angulation of $\Pi$, and $\gamma_1,\ldots,\gamma_a\in \Delta$ be $m$-diagonals ordered clockwise that are edges of an $(m+2)$-gon $P$ inside $\Pi$ determined by $\Delta$, for some $i\geq 2$. Let $\ell^k$ be the the number of edges of $P$ lying between $\gamma_{k+1}$ and $\gamma_{k}$, $k=1,\ldots, a$, where we set $\gamma_{a+1}=\gamma_1$. 

Let $Q$ be an $m$-coloured quiver associated to $\Delta$ by $\Psi$, and let $\ell_k$ be the colour of the arrow $\gamma_k\to \gamma_{k+1}$, for $k=1,\ldots, a$.

Notice that, by the skew symmetry property of $m$-coloured quivers, we have
$$
\ell_k=m-\ell^k,
$$
for all $k=1,\ldots, a$.

\begin{prop}\label{prop_colour_k_cycles}
With the same notation as above, the complete quiver on vertices $\gamma_1,\ldots, \gamma_a$ is a subquiver of $Q$. 

Furthermore, the following formulas hold.
$$
\ell_1+\ldots +\ell_a=(a-1)(m+1)-1, \hspace{0.7cm} \ell^1+\ldots +\ell^a =m-a+2
$$
\end{prop}
\begin{proof} 
We can represent the $(m+2)$-gon $\Pi$ and the $m$-diagonals $\gamma_1,\ldots, \gamma_a\in\Delta$ as follows.
\begin{center}
\includegraphics[scale=0.38]{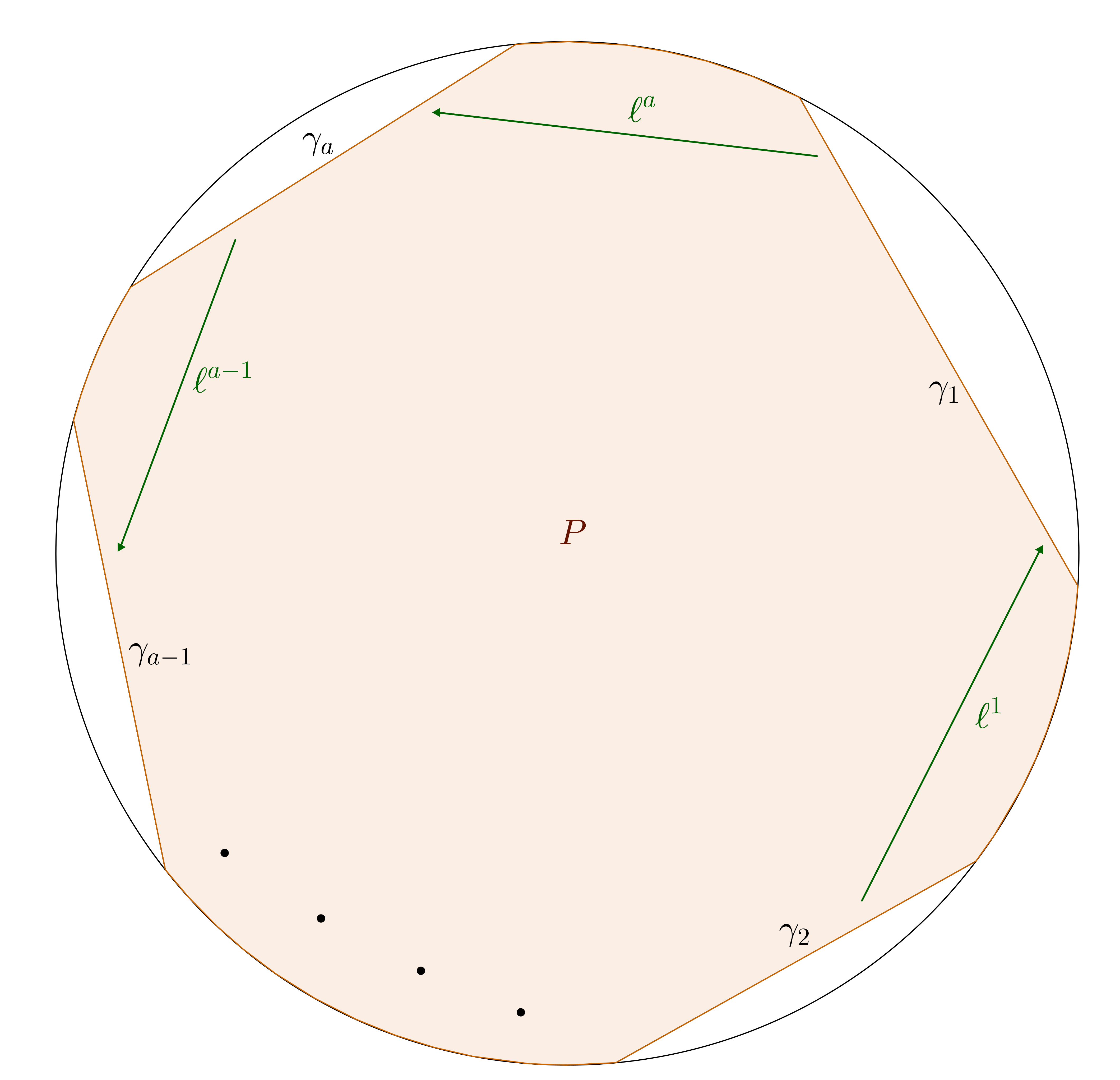}
\end{center}
Using the definition of $\Psi$, we can see that the subquiver of $Q$ associated to $\gamma_1,\ldots, \gamma_a$ by $\Psi$ is the following complete quiver on vertices $\gamma_1,\ldots, \gamma_a$,
$$
\xymatrixcolsep{0.4pc}\xymatrix{ & & \gamma_a\ar@<0.3ex>[dll]\ar@<0.3ex>[ddl]\ar@<0.3ex>[ddr]\ar@<0.3ex>[drr] & & \\
\gamma_{a-1}\ar@<0.3ex>[urr]\ar@<0.3ex>[drrr]\ar@<0.3ex>[rrrr]\ar@{.}[dr] & & & & \gamma_1\ar@<0.3ex>[ull]\ar@<0.3ex>[llll]\ar@<0.3ex>[dlll]\ar@<0.3ex>[dl] \\
& \gamma_3\ar@<0.3ex>[uur]\ar@<0.3ex>[rr]\ar@<0.3ex>[urrr] & & \gamma_2\ar@<0.3ex>[uul]\ar@<0.3ex>[ulll]\ar@<0.3ex>[ll]\ar@<0.3ex>[ur] & 
}
$$
with arrows 
$$
\xymatrix@1{\gamma_k\ar@<0.3ex>[r]^{(\ell_k)} & \gamma_{k+1}\ar@<0.3ex>[l]^{(\ell^k)}}, \hspace{1cm} k=1,\ldots, a. 
$$
Now, since $P$ is an $(m+2)$-gon, we have $\ell^1+\ldots +\ell^a+a=m+2$, that is, 
$$
\ell^1+\ldots+\ell^a=m-a+2.
$$
Also, since $\ell_k=m-\ell^k$ for all $k=1,\ldots, a$, we get:
\begin{align*}
\ell_1+\ldots +\ell_a= & (m-\ell^1)+\ldots +(m-\ell^a) \\ = & am-(\ell^1+\ldots +\ell^a)=am-(m-a+2) \\ = & (a-1)(m+1)-1.
\end{align*}
\end{proof}

\begin{ese}\label{ese_12gon_colour_of_arrows}
Let $m=2$, $n=4$, and consider the following 12-angulation of a regular dodecagon $\Pi$.
\begin{center}
\begin{figure*}[h!]
\includegraphics[scale=0.5]{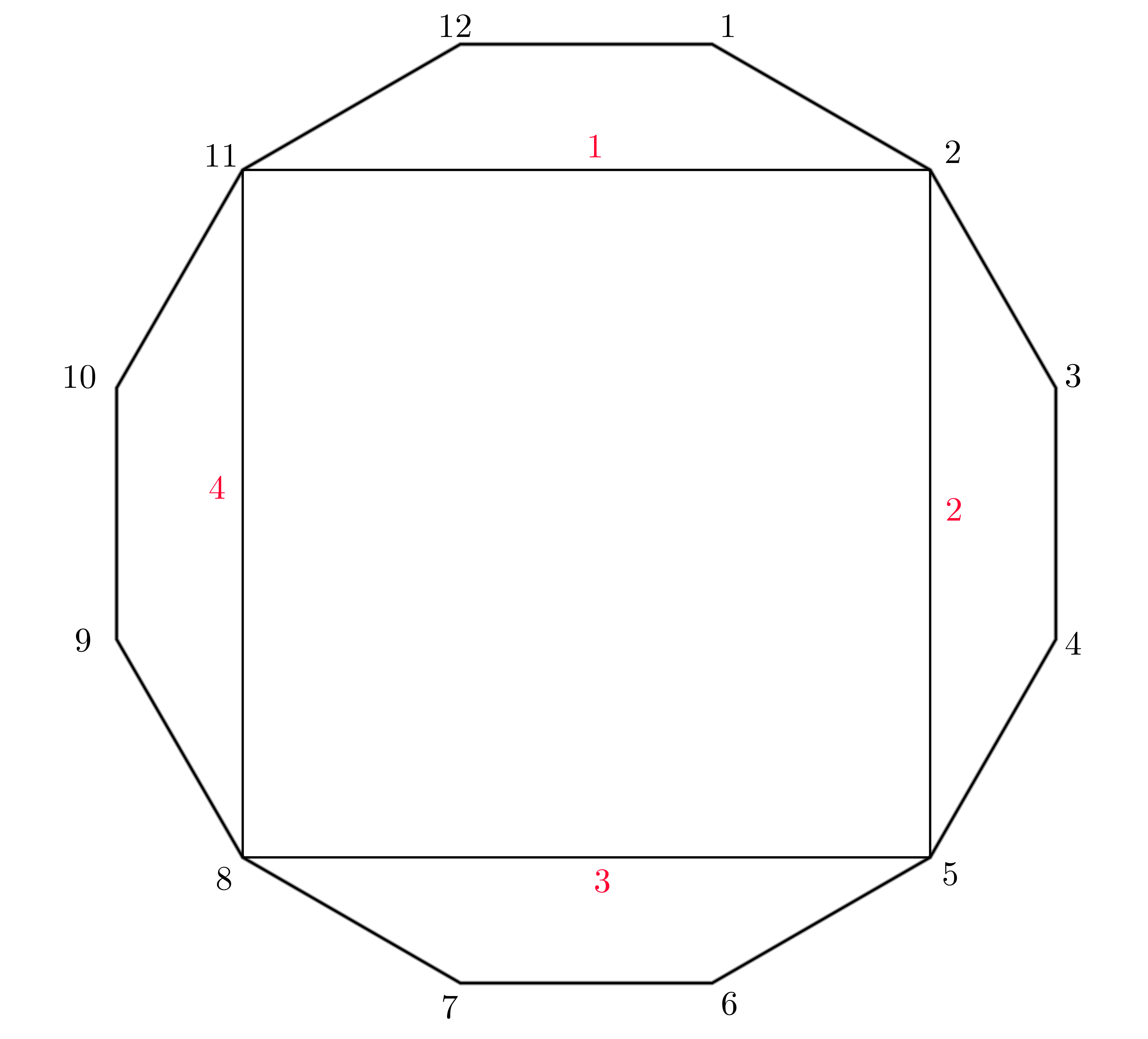}
\end{figure*}
\end{center}
The associated 2-coloured quiver $Q$ is given by
$$
\xymatrix{ & 1\ar@<0.3ex>[dr] \ar@<-3pt>@{}[dr]^{\scalemath{0.5}{(2)}}\ar@<0.3ex>[dl] \ar@<-3pt>@{}[dl]^{\scalemath{0.5}{(0)}}\ar@<0.3ex>[dd] \ar@<-1.7pt>@{}[dd]^<<<<<<<{\scalemath{0.5}{(1)}}|(.5){\hole} & \\
4\ar@<0.3ex>[rr] \ar@<-1.7pt>@{}[rr]^<<<<<<<{\scalemath{0.5}{(1)}}\ar@<0.3ex>[dr] \ar@<-3pt>@{}[dr]^{\scalemath{0.5}{(0)}}\ar@<0.3ex>[ur] \ar@<-3pt>@{}[ur]^{\scalemath{0.5}{(2)}} & & 2\ar@<0.3ex>[ll] \ar@<-1.4pt>@{}[ll]^<<<<<<<{\scalemath{0.5}{(1)}}\ar@<0.3ex>[ul] \ar@<-3pt>@{}[ul]^{\scalemath{0.5}{(0)}}\ar@<0.3ex>[dl] \ar@<-3pt>@{}[dl]^{\scalemath{0.5}{(2)}} \\
& 3\ar@<0.3ex>[uu] \ar@<-1.7pt>@{}[uu]^<<<<<<<{\scalemath{0.5}{(1)}}|(.5){\hole}\ar@<0.3ex>[ur] \ar@<-3pt>@{}[ur]^{\scalemath{0.5}{(2)}}\ar@<0.3ex>[ul] \ar@<-3pt>@{}[ul]^{\scalemath{0.5}{(0)}} &}
$$
Let $c_{i,j}$ be the colour of the arrow $i\to j$ in $Q$. Then the two formulas from Proposition \ref{prop_colour_k_cycles} hold. 

For example, if we consider the vertices 1,2,3,4, we have
$$
c_{1,4}+c_{4,3}+c_{3,2}+c_{2,1}=0+0+0+0=0=2-4+2
$$
$$
c_{1,2}+c_{2,3}+c_{3,4}+c_{4,1}=2+2+2+2=9=3\cdot 3 -1.
$$ 
Also, if we consider vertices 1,2,4, we have
$$
c_{1,4}+c_{4,2}+c_{2,1}=0+1+0=1=2-3+2
$$
$$
c_{1,2}+c_{2,4}+c_{4,1}=2+1+2=5=2\cdot 3-1.
$$
\end{ese}

We can now define in a natural way a group arising from an $m$-coloured quiver of mutation type $\overrightarrow{A_{n-1}}$.

\section{Presentations of braid groups of type $\mathbb{A}_{n-1}$}

\begin{defn}\label{defn_group_ass_to_quiver}
Let $Q$ be an $m$-coloured quiver of mutation type $\overrightarrow{A_{n-1}}$, with vertices $1,\ldots, n-1$. For an arrow $i\to j$, let $c_{ij}$ be its colour. Define the group $B_Q$ to be generated by $s_1,\ldots, s_{n-1}$ subject to the relations:
\begin{enumerate}
\item $s_i s_j=s_js_i$ if there is no arrow between $i$ and $j$ (in either direction);
\item $s_i s_j s_i=s_js_is_j$ if there is a pair of arrows $\xymatrix@1{i\ar@<0.3ex>[r] & j\ar@<0.3ex>[l]}$;
\item $s_{i_1}s_{i_2}s_{i_3}s_{i_1}=s_{i_2}s_{i_3}s_{i_1}s_{i_2}=s_{i_3}s_{i_1}s_{i_2}s_{i_3}$ if 
$$
\xymatrixcolsep{1pc}\xymatrix{ & i_1\ar@<0.3ex>[dl]\ar@<0.3ex>[dr] & \\
i_2 \ar@<0.3ex>[rr]\ar@<0.3ex>[ur] & & i_3 \ar@<0.3ex>[ll]\ar@<0.3ex>[ul]}
$$
is a subquiver of $Q$ and $c_{i_1, i_2}+c_{i_2, i_3}+c_{i_3, i_1}=2m+1$.
\end{enumerate}  
\end{defn}

\begin{rem}
Let $Q$ be an $m$-coloured quiver of mutation type $\overrightarrow{A_{n-1}}$. By Theorem \ref{thm_bij_col_quivs_m_guls} we can find an $(m+2)$-angulation $\Delta$ of $\Pi$ such that $\Psi(\Delta)=Q$. Therefore any subquiver of $Q$ of the form 
$$
\xymatrixcolsep{1pc}\xymatrix{ & i_1\ar@<0.3ex>[dl]\ar@<0.3ex>[dr] & \\
i_2 \ar@<0.3ex>[rr]\ar@<0.3ex>[ur] & & i_3 \ar@<0.3ex>[ll]\ar@<0.3ex>[ul]}
$$
corresponds to three $m$-diagonals $i_1,i_2,i_3\in\Delta$ that, by Definition of $\Psi$, are edges of the same $(m+2)$-gon identified by $\Delta$ in $\Pi$. Thus, Proposition \ref{prop_colour_k_cycles} applied in the case $i=3$ tells us that exactly one of the following holds true.
\begin{itemize}
\item The sum of the colours of the arrows going clockwise is $2m+1$;
\item The sum of the colours of the arrows going counterclockwise is $2m+1$.
\end{itemize}
Therefore any complete full subquiver of $Q$ on three vertices gives a relation in $B_Q$ of type (3) in Definition \ref{defn_group_ass_to_quiver}.
\end{rem}

\begin{rem}\label{rem_braid_group}
Consider the $m$-coloured quiver $\overrightarrow{A_{n-1}}$.
$$
\xymatrix@1{1\ar@<0.3ex>[r]^{(0)} & 2\ar@<0.3ex>[l]^{(m)}\ar@<0.3ex>[r]^{(0)} & \cdots\ar@<0.3ex>[l]^{(m)} \ar@<0.3ex>[r]^{(0)}& (n-2)\ar@<0.3ex>[l]^{(m)}\ar@<0.3ex>[r]^{(0)} & (n-1)\ar@<0.3ex>[l]^{(m)}}.
$$
Then its associated group $B_{\overrightarrow{A_{n-1}}}$ is the braid group of type $A_{n-1}$, given in terms of its standard presentation. More precisely, it is generated by $s_1,\ldots, s_{n-1}$ subject to the relations

\[
\begin{cases}
s_i s_j=s_j s_i,  & \text{ if } |i-j|>1  \\
s_i s_{i+1}s_i= s_{i+1} s_i s_{i+1} & 
\end{cases}
\]
\end{rem}

\begin{ese}
Let $m=2$, and consider the following $2$-coloured quiver $Q$ of mutation type $\overrightarrow{A_{4}}$.
$$
\xymatrixcolsep{3pc}\xymatrix{ & 2\ar@<0.3ex>[dl]\ar@<-0.3pt>@{}[dl]^{\scalemath{0.5}{(2)}} \ar@<0.3ex>[dr]\ar@<-0.3pt>@{}[dr]^{\scalemath{0.5}{(0)}} & & \\
1\ar@<0.3ex>[ur]\ar@<-0.3pt>@{}[ur]^{\scalemath{0.5}{(0)}} \ar@<0.3ex>[rr]\ar@<-0.3pt>@{}[rr]^{\scalemath{0.5}{(1)}} & & 3\ar@<0.3ex>[ll]\ar@<-0.3pt>@{}[ll]^{\scalemath{0.5}{(1)}} \ar@<0.3ex>[ul]\ar@<-0.3pt>@{}[ul]^{\scalemath{0.5}{(2)}} \ar@<0.3ex>[r]\ar@<-0.3pt>@{}[r]^{\scalemath{0.5}{(0)}} & 4\ar@<0.3ex>[l]\ar@<-0.3pt>@{}[l]^{\scalemath{0.5}{(2)}}}
$$ 
Then the group $B_Q$ associated to $Q$ is generated by $s_1, s_2, s_3, s_4$ subject to the following relations:
\begin{align*}
s_1 s_4= & s_4s_1 \\
s_2 s_4= & s_4 s_2 \\
s_1 s_2 s_1 = & s_2 s_1 s_2  \\
s_1 s_3 s_1 = & s_3 s_1 s_3 \\
s_2 s_3 s_2 = & s_3 s_2 s_3 \\
s_3 s_4 s_3 = & s_4 s_3 s_4 \\
s_1 s_3 s_2 s_1 = & s_3 s_2 s_1 s_3 = s_2 s_1 s_3 s_2. 
\end{align*}
Notice that the quiver $Q$ and the associated group $B_Q$ can also be obtained in the uncoloured case (see \cite{GM}). In Example \ref{ese_relation_following} we will give an example of a 2-coloured quiver $Q$ of mutation type $\overrightarrow{A_4}$ that does not arise from the uncoloured theory.
\end{ese}

Let $Q$ be an $m$-coloured quiver of mutation type $\overrightarrow{A_{n-1}}$, $k\in Q_0$. Let $Q'=\mu_k (Q)$ be its coloured quiver mutation at $k$. Let $\{s_i |i\in\ Q_0\}$ (resp. $\{t_i |i\in Q'_0\}$) be the generators of $B_Q$ (resp. $B_{Q'}$) introduced in Definition \ref{defn_group_ass_to_quiver}. Let $F_Q$ be the free group with generators $s_i$, $i\in Q_0$.

\begin{defn}
Let $\varphi_k : F_Q \to B_{Q'}$ be the group homomorphism defined by
$$
\varphi_k (s_i)=\begin{cases} t_k t_i t_k^{-1}, & \text{ if } \xymatrix@1{k\ar[r]^{(0)} & i} \\
t_i, & \text{ otherwise}
\end{cases}.
$$
\end{defn}

\begin{prop}
The group homomorphism defined above induces a group homomorphism $\varphi_{k}:B_Q\to B_{Q'}$.
\end{prop}
\begin{proof}
We need to show that all the relations in $B_Q$ are preserved by $\varphi_k$. 

By Theorem \ref{thm_bij_col_quivs_m_guls} we can find an $(m+2)$-angulation $\Delta$ of $\Pi$ such that $\Psi(\Delta)=Q$. 

Consider the $m$-diagonal $k\in\Delta$. By Proposition \ref{prop_properties_m2_guls}(3), there are exactly two $(m+2)$-gons $P_{k}^{(1)},P_k^{(2)}$ in $\Pi$ identified by $\Delta$ that have $k$ as an edge. Let $u_0, u_1,\ldots, u_a$ (resp. $v_0, v_1,\ldots, v_b$) be the $m$-diagonals of $\Delta$ that are edges of $P_{k}^{(1)}$ (resp. of $P_k^{(2)}$), ordered clockwise, for some $a,b\geq 0$, where we set $u_0=v_0=k$. By Proposition \ref{prop_properties_m2_guls}(3), we also know that, if $a>0$ (resp. $b>0$), then $u_a$ (resp. $v_b$) will be edge of another $(m+2)$-gon $P_{u_a}$ (resp. $P_{v_b}$) in $\Pi$ identified by $\Delta$. Let $u_a,w_1,\ldots, w_{\ell}$ (resp. $v_b, z_1, \ldots, z_{h}$) be the $m$-diagonals of $\Delta$ that are edges of $P_{u_a}$ (resp. $P_{v_b}$), ordered clockwise, for some $\ell, h\geq 0$. 

Applying the definition of $\Psi$, we get that the union of the subquivers $F_{Q,k}, F_{Q,u_a}$ and $F_{Q,v_b}$ of $Q$ is the following glueing of complete subquivers of $Q$.
$$
\xymatrixcolsep{0.4pc}\xymatrix{& w_2\ar@{.}[rr]\ar@<0.3ex>[dl]\ar@<0.3ex>[ddr]\ar@<0.3ex>[drrr] & & w_{\ell -1}\ar@<0.3ex>[dr]\ar@<0.3ex>[ddl]\ar@<0.3ex>[dlll] & & & & & \\
w_1\ar@<0.3ex>[ur]\ar@<0.3ex>[rrrr]\ar@<0.3ex>[urrr]\ar@<0.3ex>[drr] & & & & w_{\ell}\ar@<0.3ex>[llll]\ar@<0.3ex>[dll]\ar@<0.3ex>[ulll]\ar@<0.3ex>[ul] & & & & \\
& & u_a\ar@<0.3ex>[ull]\ar@<0.3ex>[uul]\ar@<0.3ex>[uur]\ar@<0.3ex>[urr]\ar@<0.3ex>[dll]\ar@<0.3ex>[ddl]\ar@<0.3ex>[ddr]\ar@<0.3ex>[drr] & & & v_1\ar@<0.3ex>[dl]\ar@<0.3ex>[rr]\ar@<0.3ex>[drrr]\ar@<0.3ex>[ddr] & & v_2\ar@{.}[dr]\ar@<0.3ex>[ddl]\ar@<0.3ex>[ll]\ar@<0.3ex>[dlll] & \\
u_{a-1}\ar@<0.3ex>[urr]\ar@<0.3ex>[rrrr]\ar@<0.3ex>[drrr]\ar@{.}[dr] & & & & k\ar@<0.3ex>[llll]\ar@<0.3ex>[ull]\ar@<0.3ex>[dlll]\ar@<0.3ex>[dl]\ar@<0.3ex>[ur]\ar@<0.3ex>[urrr]\ar@<0.3ex>[drr]\ar@<0.3ex>[rrrr] & & & & v_{b-1}\ar@<0.3ex>[llll]\ar@<0.3ex>[ulll]\ar@<0.3ex>[dll] \\
& u_2\ar@<0.3ex>[uur]\ar@<0.3ex>[urrr]\ar@<0.3ex>[rr] & & u_1\ar@<0.3ex>[ll]\ar@<0.3ex>[ulll]\ar@<0.3ex>[uul]\ar@<0.3ex>[ur] & & & v_b\ar@<0.3ex>[ull]\ar@<0.3ex>[uul]\ar@<0.3ex>[urr]\ar@<0.3ex>[uur]\ar@<0.3ex>[dll]\ar@<0.3ex>[drr]\ar@<0.3ex>[ddl]\ar@<0.3ex>[ddr] & & \\
& & & & z_h\ar@<0.3ex>[urr]\ar@<0.3ex>[dr]\ar@<0.3ex>[rrrr]\ar@<0.3ex>[drrr] & & & & z_1\ar@<0.3ex>[ull]\ar@<0.3ex>[llll]\ar@<0.3ex>[dl]\ar@<0.3ex>[dlll] \\
& & & & & z_{h-1}\ar@{.}[rr]\ar@<0.3ex>[ul]\ar@<0.3ex>[uur]\ar@<0.3ex>[urrr] & & z_2\ar@<0.3ex>[ulll]\ar@<0.3ex>[ur]\ar@<0.3ex>[uul] & } 
$$
Notice that the relations in $B_Q$ involving $s_{u_a}$ are:
\begin{enumerate}
\item $s_{u_a}s_{p}=s_ps_{u_a}$ for all $p\notin \mathcal{N}_{Q,u_a}$;
\item $s_{u_a} s_{p}s_{u_a}=s_p s_{u_a} s_p$ for all $p\in\mathcal{N}_{Q,u_a}$;
\item $s_{u_a}s_{u_i}s_{u_j}s_{u_a}=s_{u_i}s_{u_j}s_{u_a}s_{u_i}=s_{u_j}s_{u_a}s_{u_i}s_{u_j}$ for all \\$i,j\in\{0,\ldots, a-1\}$, with $i<j$;
\item $s_{u_a}s_{w_i}s_{w_j}s_{u_a}=s_{w_i}s_{w_j}s_{u_a}s_{w_i}=s_{w_j}s_{u_a}s_{w_i}s_{w_j}$ for all \\$i,j\in\{1,\ldots, \ell\}$, with $i<j$.
\end{enumerate}
The relations in $B_Q$ involving $s_{v_b}$ are analogous. 

Now, Lemma \ref{lem_local_mutations} thus tells us the result of the mutation of $Q$ at $k$. 

We need to check that the relations defining $B_Q$ are preserved by $\varphi_k$. We split the cases as in Lemma \ref{lem_local_mutations}. For convenience, let $c$ (resp. $d$) be the colour of the arrow $k\to u_a$ (resp. of the arrow $k\to v_b$). Let $\tilde{s}_p:=\varphi_k (s_p)$ for all $p\in Q_0$.
\begin{itemize}
\item[a)] If $c\neq 0\neq d$, then $\tilde{s}_p=s_p$ for all $p\in Q_0$, and the result follows trivially.
\item[b)] If $c=0$ and $d\neq 0$ then, if we apply $\mu_k$, the shape of the subquiver of $Q$ previously drawn changes as follows.

$$
\xymatrixcolsep{0.4pc}\xymatrix{& & & & w_2\ar@<0.3ex>[dl]\ar@{.}[rr]\ar@<0.3ex>[drrr]\ar@<0.3ex>[ddr] & & w_{\ell -1}\ar@<0.3ex>[dlll]\ar@<0.3ex>[ddl]\ar@<0.3ex>[dr] & & \\
& & & w_1\ar@<0.3ex>[ur]\ar@<0.3ex>[drr]\ar@<0.3ex>[urrr]\ar@<0.3ex>[rrrr] & & & & w_{\ell}\ar@<0.3ex>[llll]\ar@<0.3ex>[dll]\ar@<0.3ex>[ulll]\ar@<0.3ex>[ul] & \\
& & u_{a-1}\ar@<0.3ex>[dll]\ar@<0.3ex>[ddl]\ar@<0.3ex>[ddr]\ar@<0.3ex>[drr] & & & u_a\ar@<0.3ex>[ull]\ar@<0.3ex>[uul]\ar@<0.3ex>[uur]\ar@<0.3ex>[urr]\ar@<0.3ex>[dl]\ar@<0.3ex>[rr]\ar@<0.3ex>[drrr]\ar@<0.3ex>[ddr] & & v_1\ar@{.}[dr]\ar@<0.3ex>[ddl]\ar@<0.3ex>[ll]\ar@<0.3ex>[dlll] & \\
u_{a-2}\ar@<0.3ex>[urr]\ar@<0.3ex>[rrrr]\ar@<0.3ex>[drrr]\ar@{.}[dr] & & & & k\ar@<0.3ex>[llll]\ar@<0.3ex>[ull]\ar@<0.3ex>[dlll]\ar@<0.3ex>[dl]\ar@<0.3ex>[ur]\ar@<0.3ex>[urrr]\ar@<0.3ex>[drr]\ar@<0.3ex>[rrrr] & & & & v_{b-1}\ar@<0.3ex>[llll]\ar@<0.3ex>[ulll]\ar@<0.3ex>[dll] \\
& u_2\ar@<0.3ex>[uur]\ar@<0.3ex>[urrr]\ar@<0.3ex>[rr] & & u_1\ar@<0.3ex>[ll]\ar@<0.3ex>[ulll]\ar@<0.3ex>[uul]\ar@<0.3ex>[ur] & & & v_b\ar@<0.3ex>[ull]\ar@<0.3ex>[uul]\ar@<0.3ex>[urr]\ar@<0.3ex>[uur]\ar@<0.3ex>[dll]\ar@<0.3ex>[drr]\ar@<0.3ex>[ddl]\ar@<0.3ex>[ddr] & & \\
& & & & z_h\ar@<0.3ex>[urr]\ar@<0.3ex>[dr]\ar@<0.3ex>[rrrr]\ar@<0.3ex>[drrr] & & & & z_1\ar@<0.3ex>[ull]\ar@<0.3ex>[llll]\ar@<0.3ex>[dl]\ar@<0.3ex>[dlll] \\
& & & & & z_{h-1}\ar@{.}[rr]\ar@<0.3ex>[ul]\ar@<0.3ex>[uur]\ar@<0.3ex>[urrr] & & z_2\ar@<0.3ex>[ulll]\ar@<0.3ex>[ur]\ar@<0.3ex>[uul] & } 
$$
By definition of $\varphi_k$, we have $\tilde{s}_{u_a}=\varphi_k (s_{u_a})=t_k t_{u_a} t_k^{-1}$ and $\tilde{s}_p=\varphi_k (s_{p})=t_p$ for all $p\neq u_a$. 

It is enough to check that the relations involving $s_{u_a}$ are preserved. 
\begin{itemize}
\item[(1)] Suppose that $p\notin \mathcal{N}_{Q,u_a}$. Then:
\begin{itemize}
\item[$\bullet$] if $p\notin\mathcal{N}_{Q,k}$ then, using $t_k t_p=t_p t_k$ and $t_{u_a}t_p=t_p t_{u_a}$ we get
\begin{align*}
\tilde{s}_{u_a}\tilde{s}_p=t_k t_{u_a} t_k^{-1} t_p=t_p t_k t_{u_a} t_k^{-1}=\tilde{s}_p \tilde{s}_{u_a}.
\end{align*} 
\item[$\bullet$] if $p\in \mathcal{N}_{Q,k}$, then $p=v_i$ for some $i\in\{1,\ldots, b\}$.  Thus, using $t_k t_{u_a} t_{v_{i}} t_k=t_{v_{i}} t_k t_{u_a}t_{v_{i}}$ and $t_k t_{v_{i}}t_k=t_{v_{i}}t_k t_{v_{i}}$, we get
\begin{align*}
\tilde{s}_{u_a}\tilde{s}_{v_{i}} = & t_k t_{u_a} t_k^{-1}t_{v_{i}} = t_{v_{i}} t_k t_{u_a} t_{v_{i}} t_k^{-1}t_{v_{i}}^{-1}t_k^{-1}t_{v_{i}} \\
= & t_{v_{i}} t_k t_{u_a} t_{v_{i}} t_{v_{i}}^{-1} t_k^{-1} t_{v_{i}}^{-1}t_{v_{i}}= t_{v_{i}}t_k t_{u_a} t_k^{-1}=\tilde{s}_{v_{i}} \tilde{s}_{u_a}.
\end{align*}
\end{itemize}
\item[(2)] Suppose $p\in \mathcal{N}_{Q,u_a}$. Then:
\begin{itemize}
\item[$\bullet$] if $p\in \mathcal{N}_{Q,k}$, then $p=u_i$ for some $i\in\{0,\ldots, a-1\}$. Then, using the relations $t_{k}t_{u_a}t_{k}=t_{u_a}t_k t_{u_a}$, $t_k t_{u_i}t_k=t_{u_i}t_k t_{u_i}$ and $t_{u_a} t_{u_{i}}=t_{u_{i}}t_{u_a}$ we get
\begin{align*}
\tilde{s}_{u_a}\tilde{s}_{u_{i}}\tilde{s}_{u_a}= & t_k t_{u_a} t_k^{-1} t_{u_{i}} t_k t_{u_a} t_k^{-1}= t_{u_a}^{-1} t_k t_{u_a} t_{u_{i}}t_{u_a}^{-1} t_k t_{u_a} \\
= & t_{u_a}^{-1} t_k t_{u_a} t_{u_a}^{-1} t_{u_{i}} t_k t_{u_a} 
=  t_{u_a}^{-1} t_k t_{u_{i}} t_k t_{u_a} \\
= & t_{u_a}^{-1} t_{u_{i}} t_k t_{u_{i}} t_{u_a}= t_{u_{i}} t_{u_a}^{-1} t_k t_{u_a} t_{u_{i}} \\
= & t_{u_{i}} t_k t_{u_a} t_{k}^{-1} t_{u_{i}} = \tilde{s}_{u_{i}}\tilde{s}_{u_a} \tilde{s}_{u_{i}}.
\end{align*}
\item[$\bullet$] if $p\notin \mathcal{N}_{Q,k}$, then $p=w_i$ for some $i\in\{1,\ldots, \ell\}$. Then, using $t_k t_{w_i}=t_{w_i} t_k$ and $t_{u_a} t_{w_i} t_{u_a}= t_{w_i} t_{u_a} t_{w_i}$ we get
\begin{align*}
\tilde{s}_{u_a} \tilde{s}_{w_i} \tilde{s}_{u_a}= & t_k t_{u_a} t_k^{-1} t_{w_i} t_k t_{u_a} t_k^{-1} = t_k t_{u_a} t_k^{-1} t_k t_{w_i} t_{u_a} t_k^{-1} \\
= & t_k t_{u_a} t_{w_i} t_{u_a} t_k^{-1} = t_k t_{w_i} t_{u_a} t_{w_i} t_{k}^{-1}= t_{w_i} t_k t_{u_a} t_k^{-1}t_{w_i} \\
= & \tilde{s}_{w_i} \tilde{s}_{u_a}\tilde{s}_{w_i}.
\end{align*}
\end{itemize}
\item[(3)] If $i,j\in\{0,\ldots, a-1\}$ and $i<j$ then, using $t_k t_{u_a}t_k=t_{u_a}t_k t_{u_a}$, $t_{u_a}t_{u_i}=t_{u_i}t_{u_a}$, $t_{u_a}t_{u_j}=t_{u_j}t_{u_a}$ and $t_k t_{u_i}t_{u_j}t_k=t_{u_i}t_{u_j}t_k t_{u_i}$ we have
\begin{align*}
\tilde{s}_{u_a} \tilde{s}_{u_i}\tilde{s}_{u_j}\tilde{s}_{u_a}= & t_k t_{u_a}t_k^{-1} t_{u_i}t_{u_j}t_k t_{u_a}t_k^{-1}=t_{u_a}^{-1}t_k t_{u_a}t_{u_i}t_{u_j}t_{u_a}^{-1}t_k t_{u_a} \\
= & t_{u_a}^{-1} t_k t_{u_i}t_{u_j}t_{u_a}t_{u_a}^{-1}t_k t_{u_a}= t_{u_a}^{-1} t_k t_{u_i}t_{u_j}t_k t_{u_a} \\
= & t_{u_a}^{-1} t_{u_i}t_{u_j}t_k t_{u_i} t_{u_a} = t_{u_i}t_{u_j}t_{u_a}^{-1}t_{k}t_{u_a}t_{u_i} \\
= & t_{u_i}t_{u_j} t_k t_{u_a}t_k^{-1} t_{u_i}=\tilde{s}_{u_i}\tilde{s}_{u_j}\tilde{s}_{u_a}\tilde{s}_{u_i}.
\end{align*}
Using the relation $t_k t_{u_i}t_{u_j}t_k=t_{u_j}t_k t_{u_i}t_{u_j}$ instead of \\$t_k t_{u_i}t_{u_j}t_k=t_{u_i}t_{u_j}t_k t_{u_i}$, one can also show that
$$
\tilde{s}_{u_a} \tilde{s}_{u_i}\tilde{s}_{u_j}\tilde{s}_{u_a}=\tilde{s}_{u_j} \tilde{s}_{u_a}\tilde{s}_{u_i}\tilde{s}_{u_j}.
$$
\item[(4)] If $i,j\in\{1,\ldots, \ell\}$ and $i<j$ then, using $t_k t_{w_i}=t_{w_i}t_k$, \\ $t_k t_{w_j}=t_{w_j}t_k$ and $t_{u_a} t_{w_i}t_{w_j}t_{u_a}=t_{w_i}t_{w_j}t_{u_a}t_{w_i}$ we get
\begin{align*}
\tilde{s}_{u_a}\tilde{s}_{w_i}\tilde{s}_{w_j}\tilde{s}_{u_a}= & t_k t_{u_a}t_k^{-1}t_{w_i}t_{w_j}t_k t_{u_a}t_k^{-1}= t_k t_{u_a}t_{w_i}t_{w_j}t_{k}^{-1}t_k t_{u_a}t_{k}^{-1} \\
= & t_k t_{u_a}t_{w_i}t_{w_j} t_{u_a}t_{k}^{-1}= t_k t_{w_i}t_{w_j}t_{u_a}t_{w_i}t_k^{-1} \\
= & t_{w_i}t_{w_j}t_k t_{u_a}t_k^{-1}t_{w_i}= \tilde{s}_{w_i}\tilde{s}_{w_j}\tilde{s}_{u_a}\tilde{s}_{w_i}.
\end{align*}
Using the relation $t_{u_a} t_{w_i}t_{w_j}t_{u_a}=t_{w_j}t_{u_a}t_{w_i}t_{w_j}$ instead of \\
$t_{u_a} t_{w_i}t_{w_j}t_{u_a}=t_{w_i}t_{w_j}t_{u_a}t_{w_i}$, one can also get
$$
\tilde{s}_{u_a}\tilde{s}_{w_i}\tilde{s}_{w_j}\tilde{s}_{u_a}= \tilde{s}_{w_j}\tilde{s}_{u_a}\tilde{s}_{w_i}\tilde{s}_{w_j}.
$$
\end{itemize}
\item[c)] The case $c\neq 0$, $d=0$ is analogous to b).
\item[d)] If $c=d=0$, then $\tilde{s}_{u_a}=t_k t_{u_a}t_k^{-1}$, $\tilde{s}_{v_b}=t_k t_{v_b}t_k^{-1}$ and $\tilde{s}_p=t_p$ for all $p\neq u_a, v_b$. Again, we just need to check that the relations involving $s_{u_a}$ and $s_{v_b}$ are preserved. We will just check that the relations involving $s_{u_a}$ are preserved, since one can apply use the same methods for $s_{v_b}$.
\begin{itemize}
\item[(1)] Suppose $p\notin \mathcal{N}_{Q,u_a}$. Then:
\begin{itemize}
\item[$\bullet$] if $p\neq v_b$, then one can show that
$$
\tilde{s}_{u_a}\tilde{s}_p=\tilde{s}_p\tilde{s}_{u_a}
$$
exactly as in b(1) and b(2).
\item[$\bullet$] if $p=v_b$ then, using the relation $t_{u_a}t_{v_b}=t_{v_b} t_{u_a}$ we get
\begin{align*}
\tilde{s}_{u_a}\tilde{s}_{v_b}= & t_k t_{u_a} t_k^{-1} t_k t_{v_b} t_k^{-1} =t_k t_{u_a} t_{v_b} t_{k}^{-1} \\
= & t_k t_{v_b} t_{u_a} t_k^{-1}= t_k t_{v_b} t_k^{-1} t_k t_{u_a} t_k^{-1}=\tilde{s}_{v_b}\tilde{s}_{u_a}. 
\end{align*}
\end{itemize}
\item[(2)] If $p\in\mathcal{N}_{Q,u_a}$, $p\neq u_a$, then one can show that
$$
\tilde{s}_{u_a}\tilde{s}_p\tilde{s}_{u_a}=\tilde{s}_p\tilde{s}_{u_a}\tilde{s}_p
$$
exactly as in b(2).
\item[(3)] If $i,j\in\{0,\ldots, a-1\}$ with $i<j$, then one can show that
$$
\tilde{s}_{u_a}\tilde{s}_{u_i}\tilde{s}_{u_j}\tilde{s}_{u_a}=\tilde{s}_{u_i}\tilde{s}_{u_j}\tilde{s}_{u_a}\tilde{s}_{u_i}=\tilde{s}_{u_j}\tilde{s}_{u_a}\tilde{s}_{u_i}\tilde{s}_{u_j}
$$
exactly as in b(3).
\item[(4)] If $i,j\in\{1,\ldots, \ell\}$ with $i<j$, then one can show that
$$
\tilde{s}_{u_a}\tilde{s}_{w_i}\tilde{s}_{w_j}\tilde{s}_{u_a}=\tilde{s}_{w_i}\tilde{s}_{w_j}\tilde{s}_{u_a}\tilde{s}_{w_i}=\tilde{s}_{w_j}\tilde{s}_{u_a}\tilde{s}_{w_i}\tilde{s}_{w_j}
$$
exactly as in b(4).
\end{itemize}
\end{itemize}
\end{proof}

\begin{prop}\label{prop_conjug_neighbours}
Let $Q$ be an $m$-coloured quiver of mutation type $\overrightarrow{A_{n-1}}$, $k\in Q_0$, and $\xymatrix@1{k\ar[r] ^{(c)} & \ell}$ an arrow of colour $c$ in $Q$. Let $\{t_i^{(j)}|i\in Q_0\}$ be the generators of $B_{\mu_k^j(Q)}$ for $j=0,\ldots, m+1$. Then 
$$
\varphi_k^{j}(t_{\ell}^{(0)})=\begin{cases} t_{\ell}^{(j)}, & \text{ for } j=0,\ldots, c \\
t_k^{(j)}t_{\ell}^{(j)}(t_k^{(j)})^{-1}, & \text{ for } j=c+1,\ldots ,m+1
\end{cases}
$$ 
where we set $t_w^{(m+1)}=t_w^{(0)}$ for all $w\in Q_0$.
\end{prop}
\begin{proof}
By the equivalent definition of $m$-coloured quiver mutation given in Proposition \ref{prop_equiv_col_quiv_mut} we know that, for each application of $\mu_k$, the colour of the arrow $k\to \ell$ decreases by one. Hence, the colour of the arrow $k\to \ell$ changes as follows.

$$
\begin{array}{ccccccc}
Q & \to & \mu_k(Q) & \to & \cdots & \to & \mu_k^c(Q) \\
(\xymatrix@1{k\ar[r]^{(c)} & \ell}) & \mapsto & (\xymatrix@1{k\ar[r]^{(c-1)} & \ell})  & \mapsto & \cdots & \mapsto & (\xymatrix@1{k\ar[r]^{(0)} & \ell})
\end{array}
$$
Therefore, if we apply $\varphi_k$ on $t_{\ell}^{(0)}$ multiple times, we get

$$
\begin{array}{ccccccc}
B_Q & \to & B_{\mu_k(Q)} & \to & \cdots & \to & B_{\mu_k^c(Q)} \\ \vspace{0.4cm}
t_{\ell}^{(0)} & \mapsto & t_{\ell}^{(1)} & \mapsto & \cdots & \mapsto & t_{\ell}^{(c)}
\end{array}
$$
Thus $\varphi_k^j (t_{\ell}^{(0)})=t_{\ell}^{(j)}$ for $j=0,\ldots, c$. 

Now, the arrow $k\to \ell$ has colour 0 in $\mu_k^c(Q)$. Hence the colour of the arrow $k\to \ell$ changes as follows.  

$$
\begin{array}{ccccccc}
\mu_k^c(Q) & \to & \mu_k^{c+1}(Q) & \to & \cdots & \to & \mu_k^{m+1}(Q)=Q \\
(\xymatrix@1{k\ar[r]^{(0)} & \ell}) & \mapsto & (\xymatrix@1{k\ar[r]^{(m)} & \ell})  & \mapsto & \cdots & \mapsto & (\xymatrix@1{k\ar[r]^{(c)} & \ell})
\end{array}
$$
Therefore, if we apply $\varphi_k$ on $t_{\ell}^{(c)}$ multiple times, we get

$$
\begin{array}{ccccccc}
B_{\mu_k^c(Q)} & \to & B_{\mu_k^{c+1}(Q)} & \to & \cdots & \to & B_{\mu_k^{m+1}(Q)}=B_Q \\ \vspace{0.4cm}
t_{\ell}^{(c)} & \mapsto & t_{k}^{(c+1)}t_{\ell}^{(c+1)}(t_{k}^{(c+1)})^{-1} & \mapsto & \cdots & \mapsto & t_{k}^{(0)}t_{\ell}^{(0)}(t_k^{(0)})^{-1}
\end{array}
$$

Thus $\varphi_k^{j}(t_{\ell}^{(0)})=t_k^{(j)}t_{\ell}^{(j)}(t_k^{(j)})^{-1}$ for all $j=c+1,\ldots, m+1$.
\end{proof}

\begin{thm}\label{thm_iso_groups}
The group homomorphism $\varphi_k : B_Q\to B_{Q'}$ is an isomorphism. 
\end{thm}
\begin{proof}
Let $\{t_p | p\in Q_0\}$ be the set of generators for $B_Q$ given in Definition \ref{defn_group_ass_to_quiver}.
Consider the following composition of group homomorphisms.
$$
\phi: B_Q\xrightarrow{\varphi_k} B_{\mu_k(Q)}\xrightarrow{\varphi_k^m} B_{\mu_k^{m+1}(Q)}=B_Q
$$
Let $\ell\in\mathcal{N}_{Q,k}$. Then, using Proposition \ref{prop_conjug_neighbours}, we can say that $\phi (t_{\ell})=t_k t_{\ell} t_k^{-1}$. Furthermore, if $w\notin\mathcal{N}_{Q,k}$, then $t_w$ commutes with $t_k$. Thus $\phi(t_w)=t_w=t_k t_w t_k^{-1}$. Hence 
$$
\phi (t_p)=t_k t_p t_k^{-1} \hspace{1cm} \text{for all } p\in Q_0.
$$
So $\phi $ is conjugation by $t_k$, that is an isomorphism. Hence $\varphi_k$ is injective. 

We can use the same argument to show that the composition 
$$
\tilde{\phi}:B_{\mu_k}(Q)\xrightarrow{\varphi_k^m} B_{\mu_k^{m+1}(Q)}=B_Q \xrightarrow{\varphi_k} B_{\mu_k (Q)}
$$
is conjugation by $t_k$, and hence it is an isomorphism. This proves that $\varphi_k$ is surjective, and hence it is an isomorphism.
\end{proof}

\begin{cor}
Let $Q$ be an $m$-coloured quiver of mutation type $\overrightarrow{A_{n-1}}$, and $k\in Q_0$. Let $\{s_i |i\in Q_0\}$ (resp. $\{t_i |i\in Q_0\}$) be a generating set for $B_Q$ (resp. $B_{\mu_k}(Q)$). Then $\varphi_k^{-1}: B_{\mu_k (Q)}\to B_Q$ is computed by 
$$
\varphi_k ^{-1}(t_i)=\begin{cases} s_k^{-1} s_i s_k, & \text{ if } \xymatrix@1{k\ar[r]^{(m)} & i} \\
s_i, & \text{ otherwise}.
\end{cases}
$$
\end{cor}

\begin{cor}
Let $Q$ be an $m$-coloured quiver of mutation type $\overrightarrow{A_{n-1}}$. Then its associated group $B_Q$ is isomorphic to the braid group of type $A_{n-1}$.
\end{cor}
\begin{proof}
By Remark \ref{rem_braid_group}, we know that $B_{\overrightarrow{A_{n-1}}}$ is the braid group of type $A_{n-1}$. Hence Theorem \ref{thm_iso_groups} implies the statement.
\end{proof}

\begin{ese}\label{ese_relation_following}
Consider the 4-angulation of a regular dodecagon given in Example \ref{ese_12gon_colour_of_arrows}.
\begin{center}
\begin{figure*}
\includegraphics[scale=0.5]{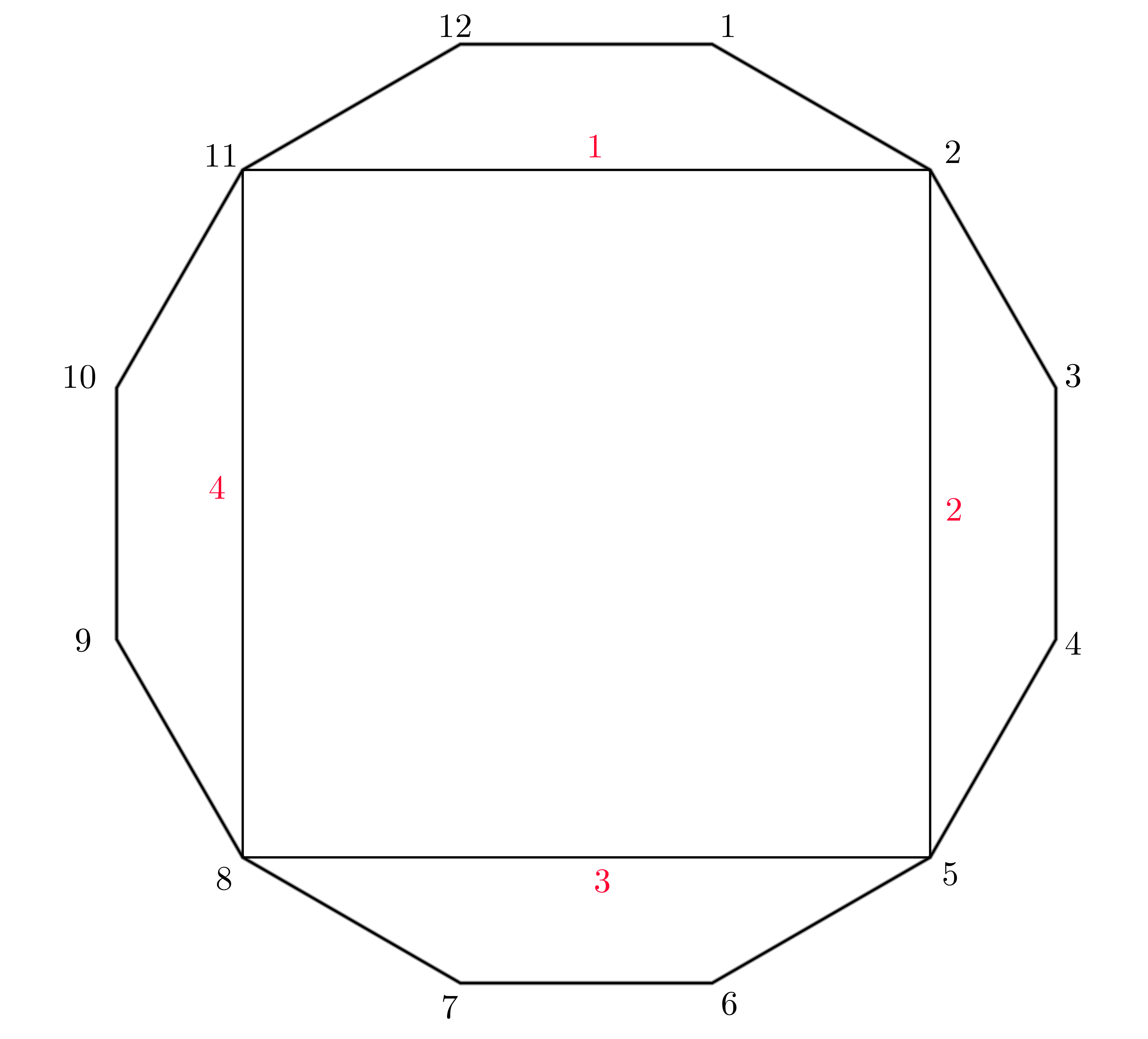}
\caption{4-angulation $\Delta$ of a regular dodecagon}
\label{im_4_gul_dodec}
\end{figure*}
\end{center}
The associated 2-coloured quiver $\Psi(\Delta)$ is the following.
$$
\xymatrix{ & 1\ar@<0.3ex>[dr] \ar@<-3pt>@{}[dr]^{\scalemath{0.5}{(2)}}\ar@<0.3ex>[dl] \ar@<-3pt>@{}[dl]^{\scalemath{0.5}{(0)}}\ar@<0.3ex>[dd] \ar@<-1.7pt>@{}[dd]^<<<<<<<{\scalemath{0.5}{(1)}}|(.5){\hole} & \\
4\ar@<0.3ex>[rr] \ar@<-1.7pt>@{}[rr]^<<<<<<<{\scalemath{0.5}{(1)}}\ar@<0.3ex>[dr] \ar@<-3pt>@{}[dr]^{\scalemath{0.5}{(0)}}\ar@<0.3ex>[ur] \ar@<-3pt>@{}[ur]^{\scalemath{0.5}{(2)}} & & 2\ar@<0.3ex>[ll] \ar@<-1.4pt>@{}[ll]^<<<<<<<{\scalemath{0.5}{(1)}}\ar@<0.3ex>[ul] \ar@<-3pt>@{}[ul]^{\scalemath{0.5}{(0)}}\ar@<0.3ex>[dl] \ar@<-3pt>@{}[dl]^{\scalemath{0.5}{(2)}} \\
& 3\ar@<0.3ex>[uu] \ar@<-1.7pt>@{}[uu]^<<<<<<<{\scalemath{0.5}{(1)}}|(.5){\hole}\ar@<0.3ex>[ur] \ar@<-3pt>@{}[ur]^{\scalemath{0.5}{(2)}}\ar@<0.3ex>[ul] \ar@<-3pt>@{}[ul]^{\scalemath{0.5}{(0)}} &}
$$

Therefore $B_{\Psi(\Delta)}$ is the group generated by $s_1,s_2,s_3,s_4$, subject to relations
\begin{align}
s_1 s_2 s_1 = & s_2 s_1 s_2 \nonumber \\
s_1 s_3 s_1 = & s_3 s_1 s_3 \nonumber \\
s_1 s_4 s_1 = & s_4 s_1 s_4 \label{eq_1} \\
s_2 s_3 s_2 = & s_3 s_2 s_3 \nonumber \\
s_2 s_4 s_2 = & s_4 s_2 s_4 \nonumber \\
s_3 s_4 s_3 = & s_4 s_3 s_4 \nonumber \\
s_1 s_2 s_3 s_1 = & s_2 s_3 s_1 s_2 = s_3 s_1 s_2 s_3 \label{eq_2}
\end{align}
\begin{align}
s_1 s_2 s_4 s_1 = & s_2 s_4 s_1 s_2 = s_4 s_1 s_2 s_4 \label{eq_3}\\
s_1 s_3 s_4 s_1 = & s_3 s_4 s_1 s_3 = s_4 s_1 s_3 s_4 \nonumber \\
s_2 s_3 s_4 s_2 = & s_3 s_4 s_2 s_3 = s_4 s_2 s_3 s_4 \nonumber
\end{align}
Notice that also cycle-type relations involving all the four vertices 1,2,3,4 hold. Specifically 
$$
s_1 s_2 s_3 s_4 s_1 = s_2 s_3 s_4 s_1 s_2 = s_3 s_4 s_1 s_2 s_3 = s_4 s_1 s_2 s_3 s_4.
$$
For example, we can get the first relation using \eqref{eq_1}, \eqref{eq_2} and \eqref{eq_3} as follows:
\begin{align*}
s_1 s_2 s_3 s_4 s_1 \stackrel{\eqref{eq_2}}{=} & s_2 s_3 s_1 s_2 s_1^{-1}s_4 s_1 \stackrel{\eqref{eq_1}}{=} s_2 s_3 s_1 s_2 s_4 s_1 s_4^{-1} \\
\stackrel{\eqref{eq_3}}{=} & s_2 s_3 s_4 s_1 s_2 s_4 s_4^{-1}= s_2 s_3 s_4 s_1 s_2.
\end{align*}
\end{ese}

The phenomenon described in Example \ref{ese_relation_following} can be generalised to subquivers of $Q$ that are complete graphs on an arbitrary number of vertices. This result can be found in Remark 3.2 of \cite{S}, but we include it together with its proof for completeness.

\begin{rem}\label{rem_less_relations}
Let $j\geq 3$, and $i_1,\ldots, i_j\in\{1,\ldots, n-1\}$, with \\
$i_1<i_2<\ldots < i_{j}$. Let $Q$ be an $m$-coloured quiver of mutation type $\overrightarrow{A_{n-1}}$, and $\{s_i| i \in Q_0\}$ the set of generators of $B_Q$ introduced in Definition \ref{defn_group_ass_to_quiver}. 

Suppose that the complete quiver on $i_1,\ldots, i_j$ is a subquiver of $Q$. 

Let $\ell_{p,q}$ be the colour of the arrow $i_p\to i_{q}$, for all $p,q\in\{1,\ldots, j\}$, $p\neq q$. Then, if 
$$
\ell_{a,b} +\ell_{b,c}+\ell_{c,a}=2m+1
$$ 
for all $a,b,c\in \{1,\ldots, j\}$ with $a<b<c$, the following relation holds in $B_Q$.
$$
s_{i_1}s_{i_2}\ldots s_{i_j}s_{i_1}=s_{i_2} s_{i_3}\ldots s_{i_j}s_{i_1}s_{i_2}. 
$$
\end{rem} 
\begin{proof}
We prove the result by induction on $j\geq 3$.
\begin{itemize}
\item If $j=3$, then the result follows directly from the definition of $B_Q$.
\item If $j>3$, then by induction hypothesis we know that 
$$
s_{i_1}s_{i_2}\ldots s_{i_{j-1}}s_{i_1}=s_{i_2} s_{i_3}\ldots s_{i_{j-1}}s_{i_1}s_{i_2}. 
$$
Using this relation together with $s_{i_1}s_{i_j}s_{i_1}=s_{i_j}s_{i_1}s_{i_j}$ and \\ $s_{i_1} s_{i_2} s_{i_j} s_{i_1}=s_{i_j} s_{i_1} s_{i_2} s_{i_j}$, we get
\begin{align*}
s_{i_1} \ldots s_{i_{j-1}} s_{i_j} s_{i_1} = & s_{i_2} \ldots s_{i_{j-1}} s_{i_1}s_{i_2}s_{i_1}^{-1}s_{i_j}s_{i_1}= s_{i_2}\ldots s_{i_{j-1}}s_{i_1}s_{i_2}s_{i_j}s_{i_1}s_{i_j}^{-1} \\
= & s_{i_2}\ldots s_{i_{j-1}}s_{i_j}s_{i_1}s_{i_2}s_{i_j}s_{i_j}^{-1}=s_{i_2}\ldots s_{i_{j-1}}s_{i_j}s_{i_1}s_{i_2}.
\end{align*}
\end{itemize}
\end{proof}

\begin{rem}
In Section 4.3 of \cite{G} the author defines a group $G_s^d$ that acts on an higher zigzag algebra of type $A$, for fixed $s,d\geq 1$. 

The construction of the group $G_s^d$ for $s=2$ is the following. Consider the quiver $Q_2^d$ given by
$$
\xymatrix{1\ar[r] & 2 \ar[d] \\
d+1\ar[u] & \reflectbox{$\ddots$} \ar[l]}.
$$
Then the group $G_2^d$ is generated by $s_1,\ldots, s_{d+1}$ subject to relations 
\begin{equation}\label{relation_high_zigzag}
s_{i_1}s_{i_2}\ldots s_{i_j}s_{i_1}=s_{i_2}\ldots s_{i_j} s_{i_1}s_{i_2}=\ldots = s_{i_j}s_{i_1}\ldots s_{i_{j-1}}s_{i_j}
\end{equation}
for all $1\leq i_1<i_2< \ldots < i_j\leq d+1$ and $j\geq 2$.

Quivers of shape $Q_2^d$ arise in my work. Indeed, if we take $m=d-1$ and $n=d+2$, then the $m$-coloured quiver associated via the map $\Psi$ to the $(m+2)$-angulation 
$$
\Delta^{(d)}=\{(1,d+1), (d+1,2d+1), (2d+1,3d+1),\ldots, (d^2+1,1)\}
$$
of a regular $(mn+2)$-gon is a complete quiver on $n-1=d+1$ vertices, and contains $Q_2^d$ as a subquiver.

For example, for $d=4$, the 5-angulation $\Delta^{(4)}$ of the 20-gon described above is given by the following
\begin{center}
\begin{figure}[h!]
\includegraphics[scale=0.5]{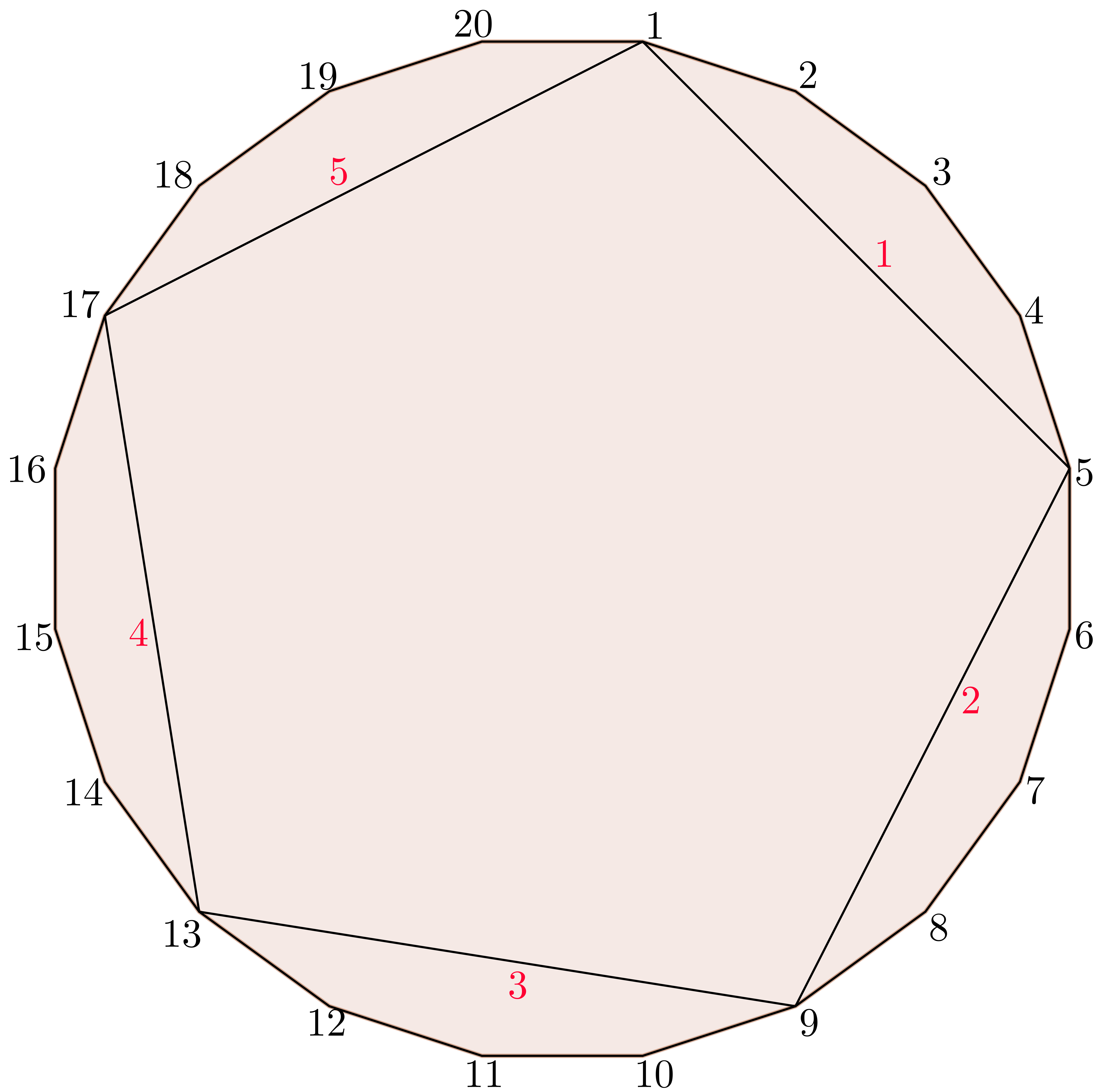}
\end{figure}
\end{center}
and the quiver associated via $\Psi$ to $\Delta^{(4)}$, this contains $Q_2^4$ as subquiver.
$$
\xymatrixcolsep{0.4pc}\xymatrixrowsep{0.7pc}\xymatrix{ & 5\ar[rr] & & 1\ar[dr] & \\
4\ar[ur] & & & & 2\ar[dll]\\
& & 3\ar[ull] & & }
$$

Therefore, Remark \ref{rem_less_relations} implies that not all the relations from \eqref{relation_high_zigzag} are necessary. The only relations one actually needs to define $G_2^d$ are
\begin{alignat*}{2}
s_{i_1} s_{i_2} s_{i_1} = &  s_{i_2} s_{i_1} s_{i_2}, & & \text{for }1\leq i_1<i_2 \leq d+1 \\
s_{i_1}s_{i_2}s_{i_3}s_{i_1}= & s_{i_2}s_{i_3}s_{i_1}s_{i_2}=s_{i_3}s_{i_1}s_{i_2}s_{i_3}, \hspace{0.8cm} & & \text{for } 1\leq i_1 <i_2 <i_3 \leq d+1.
\end{alignat*}
\end{rem}

\end{document}